\documentclass[reqno,12pt]{article}
\usepackage{amsmath, latexsym, amsfonts, amssymb, amsthm, amscd}
\usepackage[english]{babel}
\usepackage[T1]{fontenc}

\usepackage{graphicx}
\def \R{\mathbb R}

\def \E{\mathbb E}
\newcommand{\mE}[1]{\mathbb{E}\left[\, #1 \,\right]}

\usepackage{bbm}

\usepackage{color}

\newtheorem{thm}{Theorem}[section]

\newtheorem{lem}[thm]{Lemma}
\newtheorem{prop}[thm]{Proposition}
\newtheorem{defin}[thm]{Definition}
\newtheorem{rem}[thm]{Remark}

\newtheorem{example}[thm]{Example}

\newcommand{\Ind}{{\rm{\textbf{1}}}}

\renewcommand{\tilde}{\widetilde}

\newcommand{\cF}{{\ensuremath{\mathcal F}} }
\newcommand{\cG}{{\ensuremath{\mathcal G}} }

\newcommand{\cI}{{\ensuremath{\mathcal I}} }

\newcommand{\cB}{{\ensuremath{\mathcal B}} }

\newcommand{\cS}{{\ensuremath{\mathcal S}} }
\newcommand{\cW}{{\ensuremath{\mathcal W}} }

\DeclareMathSymbol{\leqslant}{\mathalpha}{AMSa}{"36} % nicer `smaller or equal'
\DeclareMathSymbol{\geqslant}{\mathalpha}{AMSa}{"3E} % nicer `larger or equal'
\DeclareMathSymbol{\eset}{\mathalpha}{AMSb}{"3F}     % nicer `emptyset'
\renewcommand{\leq}{\;\leqslant\;}                   % redef. of < or =
\renewcommand{\geq}{\;\geqslant\;}                   % redef. of > or =
\newcommand{\bbP}{{\ensuremath{\mathbb P}} }

\makeatletter
\def\captionfont@{\footnotesize}
\def\captionheadfont@{\scshape}

\long\def\@makecaption#1#2{%
  \vspace{2mm}
  \setbox\@tempboxa\vbox{\color@setgroup
    \advance\hsize-6pc\noindent
    \captionfont@\captionheadfont@#1\@xp\@ifnotempty\@xp
        {\@cdr#2\@nil}{.\captionfont@\upshape\enspace#2}%
    \unskip\kern-6pc\par
    \global\setbox\@ne\lastbox\color@endgroup}%
  \ifhbox\@ne % the normal case
    \setbox\@ne\hbox{\unhbox\@ne\unskip\unskip\unpenalty\unkern}%
  \fi
  \ifdim\wd\@tempboxa=\z@ % this means caption will fit on one line
    \setbox\@ne\hbox to\columnwidth{\hss\kern-6pc\box\@ne\hss}%
  \else % tempboxa contained more than one line
    \setbox\@ne\vbox{\unvbox\@tempboxa\parskip\z@skip
        \noindent\unhbox\@ne\advance\hsize-6pc\par}%
\fi
  \ifnum\@tempcnta<64 % if the float IS a figure...
    \addvspace\abovecaptionskip
    \moveright 3pc\box\@ne
  \else % if the float IS NOT a figure...
    \moveright 3pc\box\@ne
    \nobreak
    \vskip\belowcaptionskip
  \fi
\relax
}
\makeatother
\def\writefig#1 #2 #3 {\rlap{\kern #1 truecm
\raise #2 truecm \hbox{#3}}}

\numberwithin{equation}{section}

\usepackage[colorlinks=true, citecolor=blue]{hyperref}
\title{Stochastic neural field equations: A rigorous footing\thanks{This work was partially supported by the European Union Seventh
Framework Programme (FP7/2007-
2013) under grant agreement no. 269921 (BrainScaleS), no. 318723
(Mathemacs), and by the ERC advanced grant NerVi no. 227747.}}
\author{O. Faugeras \footnote{NeuroMathComp Team, INRIA Sophia-Antipolis. \textit{Email: olivier.faugeras@inria.fr}}\phantom{ } and J. Inglis \footnote{NeuroMathComp/TOSCA Team, INRIA Sophia-Antipolis. \textit{Email: james.inglis@inria.fr}}
}

\begin{document}
\maketitle

\begin{abstract}
We extend the theory of neural fields which has been developed in a deterministic framework by considering the influence spatio-temporal noise. The outstanding problem that we here address is the development of a theory that gives rigorous meaning to stochastic neural field equations, and conditions ensuring that they are well-posed.  Previous investigations in the field of computational and mathematical neuroscience %\cite{ bressloff:09,bressloff:10,bressloff-webber-2012,bressloff-wilkerson:12,kilpatrick-ermentrout:13} 
have been numerical for the most part. Such questions have been considered for a long time in the theory of stochastic partial differential equations, where at least two different approaches have been developed, each having its advantages and disadvantages. It turns out that both approaches have also been used in computational and mathematical neuroscience, but with much less emphasis on the underlying theory. We present a review of two existing theories and show how they can be used to put the theory of stochastic neural fields on a rigorous footing. We also provide general conditions on the parameters of the stochastic neural field equations under which we guarantee that these equations are well-posed. In so doing we relate each approach to previous work in computational and mathematical neuroscience. We hope this will provide a reference that will pave the way for future studies (both theoretical and applied) of these equations, where basic questions of existence and uniqueness will no longer be a cause for concern.
\vspace{5pt}

\noindent {\textit{Keywords:}} Stochastic neural field equations, spatially correlated noise, multiplicative noise, stochastic integro-differential equation, existence and uniqueness.
\vspace{5pt}

\noindent {\textit{AMS subject classifications:}} 60H20, 60H30, 92C20.
\end{abstract}

\section{Introduction}
Neural field equations have been widely used to study spatiotemporal dynamics of cortical regions.  Arising as continuous spatial limits of discrete models, they provide a step towards an understanding of the relationship between the macroscopic spatially 
structured activity of densely populated regions of the brain, and the underlying microscopic neural circuitry. The discrete models 
themselves describe the activity of a large number of individual neurons with no spatial dimensions. Such neural mass models have been proposed by Lopes da Silva and colleagues  \cite{lopes-da-silva-hoeks-etal:74,lopes-da-silva-rotterdam-etal:76} to account for 
oscillatory phenomena observed in the brain, and were later put on a stronger mathematical footing in the study of epileptic-like seizures in \cite{jansen-rit:95}. When taking the spatial limit of such discrete models, one typically arrives at a nonlinear integro-differential equation, in which the integral term can be seen as a nonlocal interaction term describing the spatial distribution of 
synapses in a cortical region.  Neural field models build on the original work of Wilson and Cowan \cite{wilson-cowan-biophys-1972, 
wilson-cowan-kybernetic-1973} and Amari \cite{amari-cybernetics-1977}, and are known to exhibit a rich variety of phenomena 
including stationary states, traveling wave fronts, pulses and spiral waves.  For a comprehensive review of neural field equations, 
including a description of their derivation, we refer to \cite{bressloff-JPhys-2012}.

More recently several authors have become interested in stochastic versions of neural field equations (see for example \cite{ bressloff:09,bressloff:10,bressloff-webber-2012,bressloff-wilkerson:12,kilpatrick-ermentrout:13}), in order to (amongst other things) model the effects of fluctuations on wave front propagation.  
In particular, in \cite{bressloff-webber-2012} a multiplicative stochastic term is added to the neural field equation, resulting in a stochastic nonlinear integro-differential equation of the form
\begin{equation}
\label{bressloff}
dY(t, x) = \left[ - Y(t, x) + \int_{\R} w(x,y)G(Y(t, y))dy\right]dt + \sigma(Y(t, x))dW(t, x), 
\end{equation}
for $x\in \R, t\geq0$, and some functions $G$ (referred to as the nonlinear gain function), $\sigma$ (the diffusion coefficient) and $w$ (the neural field kernel, sometimes  also called the connectivity function). Here $(W(t, x))_{x\in\R, t\geq0}$ is a stochastic process (notionally a ``Gaussian random noise'') that depends on both space and time, and which may possess some spatial correlation.

In \cite{bressloff-webber-2012} \eqref{bressloff} is used in a slightly informal way to derive some interesting phenomena.  However, from a more rigorous point of view one must be careful to understand what exactly is meant by the equation \eqref{bressloff}, and indeed, what do we understand by a solution.  The main point is that any solution must involve an object of the form
\begin{equation}
\label{stoch int}
``\int \sigma(Y(t, x))dW(t, x)"
\end{equation}
which must be precisely defined.  Of course, in the case where there is no spatial dimension, the theory of such stochastic integrals is widely disseminated, but for integrals with respect to space-time white noise (for example) it is far less well-known.  It is for this reason that we believe it be extremely worthwhile making a detailed study of how to give sense to these objects, and moreover to solutions to \eqref{bressloff} when they exist.

There are in fact two distinct approaches to defining and interpreting the quantity \eqref{stoch int}, both of which allow one to build up a theory of stochastic \textit{partial} differential equations (SPDEs).  Although \eqref{bressloff} does not strictly classify as a SPDE (since there is no derivative with respect to the spatial variable), both approaches provide a rigorous underlying theory upon which to base a study of such equations.  

The first approach generalizes the theory of stochastic processes in order to give sense to solutions of SPDEs as random processes that take their values in a Hilbert space of functions (as presented by Da Prato and Zabczyk in \cite{D-Z} and more recently by Pr\'ev\^ot and R\"ockner in \cite{Rockner}).  With this approach, the quantity \eqref{stoch int} is interpreted as a \textit{Hilbert space-valued} integral i.e. ``$\int \sigma(Y(t))dW(t)$'', where $(Y(t))_{t\geq0}$ and $(W(t))_{t\geq0}$ take their values in a Hilbert space of functions, and $\sigma(Y(t))$ is an operator between Hilbert spaces.  
The second approach is that of J. B. Walsh (as described in \cite{Walsh}), which, in contrast, takes as its starting point a PDE with a random and highly irregular ``white-noise'' term.  This approach develops integration theory with respect to a class of random measures (called martingale measures), so that \eqref{stoch int} can be interpreted as a random field in both $t$ and $x$.

In the theory of SPDEs, there are advantages and disadvantages of taking both approaches. This is also the case with regards to the stochastic neural field equation \eqref{bressloff}, as described in the conclusion below (Section \ref{conclusion}), and it is for this reason that we here develop both approaches, with the view that one or other will suit a particular reader's needs.  Taking the functional approach of Da Prato and Zabczyk is perhaps more straightforward for those with knowledge of stochastic processes, and the existing general results can be applied more directly in order to obtain, for example, existence and uniqueness.  This was the path taken in \cite{RiedlerKuehn} where the emphasis was on large deviations, though in a much less general setup than we consider here (see Remark \ref{LDP}).  However, it can certainly be argued that solutions constructed in this way are ``non-physical'', since the functional theory tends to ignore any spatial regularity properties (solutions are typically $L^2$-valued in the spatial direction).  We argue that the approach of Walsh is more suited to looking for ``physical'' solutions that are at least continuous in the spatial dimension, though we must restrict slightly the type of noise that is permitted.  A comparison of the two approaches in a general setting is presented in \cite{Dalang}, and in Section \ref{sec: comparison} in our specific setting.

The main aim of this article  is thus two fold: firstly it is to present a review of an existing theory, which is accessible to readers unfamiliar with stochastic partial differential equations, that puts the study of stochastic neural field equations on a rigorous mathematical footing. 
Secondly, we will give general conditions on the functions $G$, $\sigma$ and $w$ that are certainly satisfied for most typical choices, and under which we guarantee that there exists a solution to the neural field equation \eqref{bressloff} in some sense.  We hope this will provide a reference that will pave the way for future studies (both theoretical and applied) of these equations, where basic questions of existence and uniqueness will no longer be a cause for concern.

The layout of the article is as follows.  We first present in Section \ref{EE} the necessary material in order to consider the stochastic neural field equation \eqref{bressloff} as an evolution equation in a Hilbert space.  This involves introducing the notion of a $Q$-Wiener process taking values in a Hilbert space and stochastic integration with respect to $Q$-Wiener processes, before quoting a general existence and uniqueness result for solutions of stochastic evolution equations from \cite{D-Z}.  This theorem is then applied in Section \ref{Application} to yield a unique solution to \eqref{bressloff} interpreted as a Hilbert space-valued process, both in the case when the noise has a spatial correlation and when it does not.  The second part of the paper switches tack, and describes Walsh's theory of stochastic integration with respect to martingale measures (Section \ref{Walsh}), with a view of giving sense to a solution to \eqref{bressloff} as a random field in both time and space.  To avoid dealing with distribution-valued solutions, we in fact consider a Gaussian noise that is smoothed in the spatial direction (Section \ref{smoothed noise}), and show that, under some weak conditions, the neural field equation driven by such a smoothed noise has a unique solution in the sense of Walsh that is continuous in both time and space (Section \ref{NF smoothed GN}).  We finish with a comparison of the two approaches in Section \ref{sec: comparison}, and summarize our findings in a conclusion (Section \ref{conclusion}).

\vspace{0.2cm}
\noindent \textbf{Notation:} Throughout the article $(\Omega, \cF, (\cF_t)_{t\geq 0}, \bbP)$ will be a filtered probability space, where the filtration $(\cF_t)_{t\geq 0}$ satisfies the usual conditions (i.e. complete and right-continuous), and $L^2(\Omega, \mathcal{F}, \mathbb{P})$ will be the space of square-integrable random variables on $(\Omega, \mathcal{F}, \mathbb{P})$.  We will use the standard notation $\mathcal{B}(\mathcal{T})$ to denote the Borel $\sigma$-algebra on $\mathcal{T}$ for any topological space $\mathcal{T}$.  The Lebesgue space of $p$-integrable (with respect to the Lebesgue measure) functions over $\R^N$ for $N\in\mathbb{N} = \{1, 2, \dots\}$ will be denoted by $L^p(\R^N)$, $p\geq 1$, as usual, while $L^p(\R^N, \rho)$, $p\geq 1$, will be the Lebesgue space weighted by a measurable function $\rho:\R^N\to\R^N$.

\section{Stochastic neural field equations as evolution equations in Hilbert spaces}
\label{EE}

In this section we will need the following operator spaces.  Let $U$ and $H$ be two separable Hilbert spaces. We will write $L_0(U, H)$ to denote the space of all bounded linear operators form $U$ to $H$ with the usual norm (with the shorthand $L_0(H)$ when $U=H$), and $L_2(U, H)$ for the space of all Hilbert-Schmidt operators from $U$ to $H$, i.e. those bounded linear operators $B:U \to H$ such that 
\[
\sum_{k\geq 1} \|B(e_k)\|_{H}^2 <\infty,
\]
for some (and hence all) complete orthonormal systems $\{e_k\}_{k\geq1}$ of $U$.  Finally, a bounded linear operator $Q:U \to U$ will be said to be trace-class if $\mathrm{Tr}(Q) := \sum_{k\geq 1}\langle Q (e_k), e_k\rangle_U < \infty$, again for some (and hence all) complete orthonormal systems $\{e_k\}_{k\geq1}$ of $U$.

\subsection{Hilbert space valued $Q$-Wiener processes}

Let $U$ be a separable Hilbert space and $Q:U \to U$ a non-negative, symmetric bounded linear operator on $U$ such that $\mathrm{Tr}(Q) <\infty$.
\begin{defin}
A $U$-valued stochastic process $W = (W(t))_{t\geq0}$ is called a $Q$-Wiener process on $U$ with respect to $(\cF_t)_{t\geq0}$ if: (i) $W(0) =0$; (ii) $t \mapsto W(t)$ is continuous as a map from $[0, \infty) \to U$; (iii) $(W(t))_{t\geq0}$ is adapted to $(\cF_t)_{t\geq0}$; (iv) $W$ has independent increments; and (v) for all $0\leq s\leq t$ the law of $W(t) - W(s)$ on $U$ is Gaussian with mean $0$ and covariance operator $(t-s)Q$.
\end{defin}

Since $Q$ is non-negative and trace-class, there exists a complete orthonormal basis $\{e_k\}_{k\geq 1}$ for $U$ and a sequence of non-negative real numbers $(\lambda_k)_{k\geq1}$ such that $Qe_k = \lambda_k e_k$ and 
\[
\sum_{k=1}^\infty \lambda_k < \infty.
\]
By \cite[Proposition 4.1]{D-Z}, for arbitrary $t\geq 0$, $W$ has the expansion
\begin{align}
\label{Q-Wiener expansion}
W(t) = \sum_{k=1}^\infty \sqrt{\lambda _k}\beta_k(t) e_k, 
\end{align}
where $(\beta_k(t))_{t\geq0}$, $k=1, 2, \dots$ are mutually independent standard real-valued Brownian motions on $(\Omega, \mathcal{F}, \mathbb{P})$, and the series is convergent in $L^2(\Omega, \mathcal{F}, \mathbb{P})$.

\subsection{Stochastic integration with respect to $Q$-Wiener processes}
\label{int Q-Wiener}
Again let $U$ be a separable Hilbert space, $Q:U \to U$ a non-negative, symmetric bounded linear operator on $U$ such that $\mathrm{Tr}(Q) <\infty$, and $W = (W(t))_{t\geq0}$ be a $Q$-Wiener process on $U$ with respect to $(\cF_t)_{t\geq0}$ (given by \eqref{Q-Wiener expansion}).  

%We also let $(\mathcal{F}_t)_{t\geq0}$ be the filtration generated by $W$ i.e. $\mathcal{F}_t = \sigma(W(s): s\leq t)$ for all $t\geq0$, so that $(\mathcal{F}_t)_{t\geq0}$ is the smallest filtration with respect to which $(W(t))_{t\geq0}$ is adapted).

Let $H$ be another separable Hilbert space, and let $Q^\frac{1}{2}(U)$ be the subspace of $U$, which is a Hilbert space under the inner product
\[
\langle u, v \rangle_{Q^\frac{1}{2}(U)} := \langle Q^{-\frac{1}{2}}u, Q^{-\frac{1}{2}}v \rangle_{U}, \quad u, v \in Q^{\frac{1}{2}}(U).
\]
The space $L_2(Q^\frac{1}{2}(U), H)$ of all Hilbert-Schmidt operators from $Q^\frac{1}{2}(U)$ into $H$ plays an important role in the theory of stochastic evolution equations, and for this reason we detail the following trivial example:
%
%i.e. $B\in L_2(Q^\frac{1}{2}(U), H)$ if and only if
%\[
%\|B\|^2_{L_2(Q^\frac{1}{2}(U), H)} := \sum_{k=1}^\infty \|B(Q^\frac{1}{2}(e_k))\|^2_H < \infty,
%\]
%for some complete orthonormal basis $(e_k)_{k\geq 1}$ for $U$.

\begin{example}
\label{bounded is HS}
Let $B: U \to H$ be a bounded linear operator from $U$ to $H$ i.e. $B\in L_0(U, H)$.  Then 
\begin{align*}
\|B\|^2_{L_2(Q^\frac{1}{2}(U), H)} &= \sum_{k=1}^\infty\|B (Q^\frac{1}{2}(e_k))\|^2_{H}\\
&\leq \|B\|^2_{L_0(U, H)}\sum_{k=1}^\infty\|Q^\frac{1}{2}(e_k)\|^2_{U}\\
&=\|B\|^2_{L_0(U, H)}\sum_{k=1}^\infty \langle Q(e_k), e_k \rangle_U = \|B\|^2_{L_0(U, H)}{\rm Tr}(Q)  <\infty,
\end{align*}
since $\mathrm{Tr}(Q)<\infty$, where $\{e_k\}_{k\geq1}$ is again a complete orthonormal system for $U$ .  In other words $B\in L_0(U, H) \Rightarrow B\in L_2(Q^\frac{1}{2}(U), H)$.
\end{example}

Let $T>0$ be arbitrary. By following the construction detailed in Chapter 4 of \cite{D-Z}, we have that for a process $(\Phi(t))_{t\in [0, T]}$ the integral
\begin{equation}
\label{int}
\int_0^t\Phi(s)dW(s)
\end{equation}
has a sense as an element of $H$ when $\Phi(t)\in L_2(Q^\frac{1}{2}(U), H), t\in [0, T]$, is predictable (with respect to the filtration $(\mathcal{F}_t)_{t\geq0}$) and if 
\[
\mathbb{P}\left( \int_0^T \|\Phi(s)\|^2_{L_2(Q^\frac{1}{2}(U), H)} ds < \infty \right) = 1.
\]

Thus, for example, we have that
\[
\int_0^t B dW(t)
\]
has a sense in $H$ if and only if 
\[
\|B\|^2_{L_2(Q^\frac{1}{2}(U), H)} < \infty.
\]
When $B:U \to H$ is bounded, this certainly holds by the previous example.\\
%\olivier{This is a detail, but I do not see the reason for taking the integral with respect to $s$ since $B$ is not a function of time.} \james{I agree that there is no reason to take the integral, but I was just trying to emphasize that we are checking the condition about (which involves the integral) in the case when $\Phi(s) = B$ for all $s$.  The integral can of course be removed if you think it is clear.}

\subsection{Solutions to stochastic evolution equations}

 Let $U$ and $H$ be two separable Hilbert spaces. Consider the stochastic evolution equation
  \begin{equation}
 \label{first sde}
 dY(t) =  \left(\mathbf{A}Y(t) + \mathbf{F}(t, Y(t))\right)dt  + \mathbf{B}(t, Y(t))dW(t), \qquad Y(0) = Y_0\in H,
 \end{equation}
where $W$ is a $U$-valued $Q$-Wiener process with respect to $(\cF_t)_{t\geq0}$, with $Q:U \to U$ a non-negative, symmetric bounded linear operator on $U$ such that $\mathrm{Tr}(Q) <\infty$ as above.  

We use $\mathcal{G}_t$ to denote the predictable $\sigma$-field on $[0, t]\times\Omega$ i.e. $\mathcal{G}_t$ is the $\sigma$-algebra generated by all left-continuous stochastic processes on $[0, t]$ adapted to $(\cF_s)_{s\in[0, t]}$ for all $t\geq0$.

Fix an arbitrary finite horizon $T>0$.  We work with the following hypotheses.
\begin{itemize}
\item [{\rm (H1)}] $\mathbf{A}$ is the generator of a strongly continuous semigroup $S(t) = e^{t\mathbf{A}}, t\geq 0$, in $H$.
%\olivier{Is it important that $A$ does not depend on time?} \james{Interesting - I imagine that nothing would change if for example it was depended on time in a bounded and regularly way, but I will have to check the reference...}
\item [{\rm (H2)}] The mapping $\mathbf{F}: [0, T] \times \Omega\times H \to H$ is measurable from $([0, T]\times \Omega\times H, \cG_T\times \cB(H))$ into $(H, \cB(H))$.
\item [{\rm (H3)}] The mapping $\mathbf{B}:[0, T]\times\Omega\times H \to L_2(Q^\frac{1}{2}(U), H)$ is measurable from  $([0, T]\times\Omega\times H, \cG_T\times \cB(H))$ into $(L_2(Q^\frac{1}{2}(U), H), \cB(L_2(Q^\frac{1}{2}(U), H)))$.
\item [{\rm (H4)}] There exists a constant $C$ such that 
\[
\|\mathbf{F}(t, \omega, g) - \mathbf{F}(t, \omega, h)\|_H + \|\mathbf{B}(t, \omega, g) - \mathbf{B}(t, \omega, h)\|_{L_2(Q^\frac{1}{2}(U), H)} \leq C\|g-h\|_H,
\]
and 
\[
\|\mathbf{F}(t, \omega, h)\|^2_H + \|\mathbf{B}(t,\omega, h)\|^2_{L_2(Q^\frac{1}{2}(U), H)} \leq C^2(1+\|h\|^2_H),
\]
for all $g,h \in H$, $t\in[0, T]$ and $\omega\in\Omega$.
%\item [{\rm (v)}] for some $t>0$
%\[
%\int_0^t\|S(s)\|^2_{L_2(H)}ds < \infty;
%\]
%\item [{\rm (vi)}] for some $t>0$ and $\alpha \in (0, 1/2)$
%\[
%\int_0^ts^{-2\alpha}\|S(s)\|^2_{L_2(H)}ds < \infty;
%\]
%for all $g,h \in H$ and $t\in[0,T]$.
\end{itemize}

We now make precise what we mean by a \textit{mild} solution to \eqref{first sde}.

\begin{defin}[Mild solution]
\label{defin: DZ solution}
A predictable $H$-valued process $(Y(t))_{t\in [0, T]}$ is said to be a mild solution of \eqref{first sde} on $[0, T]$ if 
\begin{equation}
\label{P-finiteness}
\mathbb{P}\left(\int_0^T \|Y(s)\|^2_Hds < \infty \right) = 1
\end{equation}
and, for arbitrary $t\in [0, T]$, we have
\[
Y(t) = S(t)Y_0 + \int_0^tS(t-s)\mathbf{F}(s, Y(s))ds + \int_0^tS(t-s)\mathbf{B}(s, Y(s))dW(s), \quad \mathbb{P}-a.s.
\]
Note that under the conditions {\rm (H1)} - {\rm (H4)}, \eqref{P-finiteness} implies that the integrals in this expression are well-defined.  
\end{defin}

The following existence and uniqueness result is quoted from \cite{D-Z} (Theorem 7.4).

\begin{thm}[Da Prato - Zabczyk]
\label{D-Z:e+u}
Assume that conditions {\rm (H1)} - {\rm (H4)} are satisfied, and that $Y_0$ is an $\cF_0$-measurable $H$-valued random variable with finite $p$-moments for all $p\geq2$.  Then there exists a unique (up to equivalence in the Hilbert space $H$) mild solution $(Y(t))_{t\in[0, T]}$ of \eqref{first sde}. Moreover, it has a continuous modification.    

In addition, for all $p\geq 2$, there exists a constant $C^{(p )}_T>0$ such that 
\begin{equation}
\label{sup bound 1}
\sup_{t\in[0, T]} \mE{\|Y(t)\|_H^p} \leq C^{( p)}_{T}\left(1 + \mE{\|Y_0\|^p_H}\right),
\end{equation}
and for all $p>2$
\begin{equation}
\label{sup bound 2}
 \mE{\sup_{t\in[0, T]}\|Y(t)\|_H^p} \leq C^{(p )}_T\left(1 + \mE{\|Y_0\|^p_H}\right).
\end{equation}
\end{thm}

\subsection{The stochastic neural field equation: existence and uniqueness of a Hilbert-space valued solution}
\label{Application}
In this section we describe our precise interpretation of the stochastic neural field equation \eqref{bressloff} in the language of Hilbert space valued stochastic evolution equations (equation \eqref{NF SEE} below), and study existence and uniqueness properties of this equation.  Note that as opposed to \eqref{bressloff}, we here work in the more general setup when the underlying space is $N$-dimensional.

Let $\rho:\R^N\to\R^N$ be in $L^\infty({\R^N})$.  Consider the stochastic evolution equation
\begin{equation}
\label{NF SEE}
 dY(t) =  \left(-Y(t) + \mathbf{F}(Y(t))\right)dt  + \sigma(Y(t))\circ BdW(t), \quad Y(0)= Y_0 \in  L^2(\R^N, \rho),
 \end{equation}
where $\circ$ indicates the composition of operators, $W$ is an $L^2(\R^N)$-valued $Q$-Wiener process with respect to $(\cF_t)_{t\geq0}$, with $Q$ a non-negative, symmetric bounded linear operator on $L^2(\R^N)$ such that $\mathrm{Tr}(Q) <\infty$, as usual.  
Here
\begin{itemize}
 \item $B:L^2(\R^N) \to L^2(\R^N,\rho)$ is defined by
 \begin{align}
\label{B}
B(u)(x) = \int_{\R^N} \varphi(x-y)u(y)dy, \qquad x\in\R^N,\ u\in L^2(\R^N),
\end{align}
for some $\varphi \in L^1(\R^N)$\footnote{This is well-defined since $\| B(u) \|_{L^2(\R^N,\rho)} \leq \| \rho \|_{L^\infty(\R^N)}^{1/2}\|u\|_{L^2(\R^N)} \|\varphi \|_{L^1(\R^N)}$, for all $u\in L^2(\R^N)$.};
\item $\sigma: L^2(\R^N, \rho) \to L_0(L^2(\R^N, \rho))$ is such that
\[
\|\sigma(g) - \sigma(h)\|_{L_0(L^2(\R^N, \rho))} \leq C_\sigma\|g-h\|_{L^2(\R^N, \rho)}, \qquad g, h \in L^2(\R^N, \rho);
\]
\item $\mathbf{F}$ is an operator on $L^2(\R^N, \rho)$ defined by
\begin{align}
\label{F}
\mathbf{F}(h)(x) = \int_{\R^N}w(x, y)G(h(y))dy, \quad x\in \R^N,\ h\in L^2(\R^N, \rho),
\end{align}
where $w:\R^N\times\R^N\to\R$ is the neural field kernel, and $G:\R \to \R$ is the nonlinear gain function, assumed to be bounded and globally Lipschitz i.e such that there exists a constant $C_G$ with $\sup_{a\in\R}|G(a)| \leq C_G$  and 
\[
|G(a) - G(b)| \leq C_G|x- y|, \qquad \forall a, b \in \R.
\]
 \end{itemize}

Typically the nonlinear gain function $G$ is taken to be a sigmoid function, for example $G(a) = (1+e^{-a})^{-1}$, $a\in\R$.  

Of particular interest to us are the conditions on the neural field kernel $w$ which will allow us to prove existence and uniqueness of a solution taking its values in the space $L^2(\R^N, \rho)$ for some $\rho$ through Theorem \ref{D-Z:e+u}.

In \cite[footnote 1] {RiedlerKuehn} it is suggested that the condition 
\begin{equation}\label{eq:C1}
\tag{\textbf {C1}}
\int_{\R^N}\int_{\R^N} |w(x, y)|^2 dx dy <\infty
\end{equation}
together with symmetry of $w$ is enough to ensure that there exists a unique $L^2(\R^N)$-valued solution to \eqref{NF SEE}.  However, the problem is that it does not follow from (\textbf {C1}) that the operator $\mathbf{F}$ is stable on the space $L^2(\R^N)$.  For instance, suppose that in fact $G \equiv 1$ (so that $G$ is trivially globally Lipschitz).  Then for $h\in L^2(\R^N)$ (and assuming $w\geq0$) we have that
\begin{equation}
\label{F stable}
\|\textbf{F}(h)\|^2_{L^2(\R^N)} = \int_{\R^N}\|w(x, \cdot)\|_{L^1(\R^N)}^2dx.
\end{equation}

The point is that we can chose positive $w$ such that (\textbf {C1}) holds, while \eqref{F stable} is not finite.  For example in the case $N=1$ we could take $w(x, y) = (1+|x|)^{-1}(1+|y|)^{-1}$ for $x, y \in \R$.  In such a case the equation \eqref{NF SEE} is ill-posed: if $Y(t)\in L^2(\R)$ then $F(t, Y(t))$ is not guaranteed to be in $L^2(\R)$, which in turn implies that $Y(t)\not\in L^2(\R)$!

With this in mind we argue two points.  Firstly, if we want a solution in $L^2({\R^N})$, we must make the additional strong assumption that
\begin{equation}\label{eq:C2}
\tag{\textbf {C2}}
\forall x\in{\R^N}\ (y\mapsto w(x, y))\in L^1({\R^N}), \quad \mathrm{and}\quad \|w(x, \cdot)\|_{L^1({\R^N})} \in L^2({\R^N}).
\end{equation}
Indeed, below we will show that $\textbf{(C1)}$ together with $\textbf{(C2)}$ are enough to yield the existence of a unique $L^2({\R^N})$-valued solution to \eqref{NF SEE}.  

On the other hand, if we don't want to make the strong assumptions that (\textbf{C1}) and (\textbf{C2}) hold, then we have to work instead in a weighted space $L^2({\R^N}, \rho)$, in order to ensure that $\mathbf{F}$ is stable.  
In this case, we will see that if
\begin{equation}\label{eq:C1'}
\tag{\textbf {C1'}}
\exists\ \rho_w\in L^1({\R^N})\cap L^\infty({\R^N}),\quad \mathrm{s.t.}\quad \int_{\R^N} |w(x, y)|\rho_w(x)dx \leq \Lambda_w\rho_w(y)\ \forall y\in{\R^N},
\end{equation}
for some $\Lambda_w>0$, and
\begin{equation}\label{eq:C2'}
\tag{\textbf {C2'}}
\forall x\in{\R^N}\ (y\mapsto w(x, y))\in L^1({\R^N}), \quad \mathrm{and}\quad \sup_{x\in{\R^N}}\|w(x,\cdot)\|_{L^1({\R^N})} \leq C_w
\end{equation}
for some constant $C_w$, then we can prove the existence of a unique $L^2({\R^N}, \rho_w)$-valued solution to \eqref{NF SEE}.  

Condition $(\textbf{C1'})$ is in fact a non-trivial eigenvalue problem, and it is not straightforward to see whether it is satisfied for a given function $w$.  However, we chose to state the theorem below in a general way, and then below provide some important examples of when it can be applied.

In particular, one case of interest is when $w$ is homogeneous i.e. $w(x, y) = w(x-y)$ for all $x, y \in {\R^N}$, with $w\in L^1({\R^N})$.  This is an especially important case, since the homogeneity of $w$ is a very common assumption that is made in the literature (see for example \cite{bressloff-folias:04, bressloff-webber-2012, bressloff-wilkerson:12, folias-bressloff:04, kilpatrick-ermentrout:13, owen-laing-etal:07}).  However, when $w$ is homogeneous it is clear that neither (\textbf{C1}) nor (\textbf{C2}) are satisfied, and so we instead must try to show that (\textbf{C1'}) is satisfied ((\textbf{C2'}) trivially holds).  This is done in the second example below.

\begin{rem}
If we replace the spatial coordinate space ${\R^N}$ by a bounded domain $\mathcal{D}\subset {\R^N}$, so that the neural field equation \eqref{NF SEE} describes the activity of a neuron found at position $x\in\mathcal{D}$ then these kinds of issues do not come into play, and everything becomes rather trivial (under appropriate boundary conditions). Indeed, in this case one can then check the conditions of Theorem \ref{D-Z:e+u} to see that there exists a unique $L^2(\mathcal{D})$-valued solution to \eqref{NF SEE} under the condition {\rm ({\bf C2'})} only (with ${\R^N}$ replaced by $\mathcal{D}$).
%\begin{equation}
%\sup_{x\in\mathcal{D}}\|w(x,\cdot)\|_{L^1(\mathcal{D})} < \infty, \qquad \sup_{y\in\mathcal{D}}\|w(\cdot, y)\|_{L^1(\mathcal{D})} <\infty.
%\end{equation}
Although working in a bounded domain seems more physical (since any physical section of cortex is clearly bounded), the unbounded case is still often used, see \cite{bressloff-webber-2012} or the review \cite{bressloff-JPhys-2012}, and is mathematically more interesting.  The problem in passing to the unbounded case stems from the fact that the nonlocal term in \eqref{NF SEE} naturally `lives' in the space of bounded functions, while the noise naturally lives in an $L^2$ space.  These are not compatible when the underlying space is unbounded.
\end{rem}

\begin{thm}
\label{NF SEE:e+u}
Suppose that the neural field kernel $w$ either
\begin{itemize}
\item[(i)] satisfies conditions {\rm (\textbf{C1})} and {\rm (\textbf{C2})}; or 
\item[(ii)] satisfies conditions {\rm (\textbf{C1'})} and {\rm (\textbf{C2'})}.  
\end{itemize}
If (i) holds set $\rho_w\equiv 1$, while if (ii) holds let $\rho_w$ be the function appearing in condition {\rm (\textbf{C1'})}.  

Then, whenever $Y_0$ is an $\mathcal{F}_0$-measurable $L^2({\R^N}, \rho_w)$-valued random variable with finite $p$-moments for all $p\geq2$, the neural field equation \eqref{NF SEE} has a unique mild solution taking values in the space $L^2({\R^N}, \rho_w)$.  To be precise, there exists a unique $L^2({\R^N}, \rho_w)$-valued process $(Y(t))_{t\geq0}$ such that for all $T>0$
\[
\mathbb{P}\left(\int_0^T \|Y(s)\|^2_{L^2({\R^N}, \rho_w)}ds < \infty \right) = 1
\]
and,
\[
Y(t) = e^{-t}Y_0 + \int_0^te^{-(t-s)}\mathbf{F}(Y(s))ds +\int_0^te^{-(t-s)} \sigma(Y(s))\circ BdW(s), \quad \mathbb{P}-a.s.
\]
Moreover, $(Y(t))_{t\geq0}$ has a continuous modification, and satisfies the bounds \eqref{sup bound 1} and \eqref{sup bound 2} for every $T>0$ (with $H = L^2(\R^N, \rho_w)$).
\end{thm}

\begin{proof}
We check the hypotheses (H1)-(H4) in both cases $(i)$ and $(ii)$ in order to be able to apply Theorem \ref{D-Z:e+u}, with $U=L^2({\R^N})$ and $H=L^2({\R^N}, \rho_w)$.

\vspace{0.2cm}
\noindent (H1):  In our case $\mathbf{A} = -\mathbf{Id}$, and so (H1) is trivially satisfied in both cases.
 
\vspace{0.2cm}
\noindent (H2):  We check that the function $\mathbf{F}: L^2({\R^N},\rho_w) \to L^2({\R^N},\rho_w)$.  In case $(i)$ this holds since $\rho_w\equiv 1$ and for any $h\in L^2({\R^N})$
\begin{align*}
\|\mathbf{F}(h)\|^2_{ L^2({\R^N})} &= \int_{{\R^N}}\left|\int_{{\R^N}}w(x,y)G(h(y))dy\right|^2dx\\
&\leq C_G^2 \int_{{\R^N}}\|w(x, \cdot)\|_{L^1({\R^N})}^2dx < \infty,
\end{align*}
by assumption (\textbf{C2}).  Similarly in case $(ii)$ for any $h\in L^2({\R^N}, \rho_w)$
\begin{align*}
\|\mathbf{F}(h)\|^2_{ L^2({\R^N}, \rho_w)} &= \int_{{\R^N}}\left|\int_{{\R^N}}w(x,y)G(h(y))dy\right|^2 \rho_w(x)dx\\
&\leq C_G^2 \sup_{x\in{\R^N}}\|w(x, \cdot)\|_{L^1({\R^N})}^2\|\rho_w\|_{L^1({\R^N})} < \infty.
\end{align*}
Hence in either case $\mathbf{F}$ in fact maps $L^2({\R^N}, \rho_w)$ into a metric ball in $L^2({\R^N}, \rho_w)$. 

\vspace{0.2cm}
\noindent (H3): To show (H3) in both cases it suffices to check that for any $h\in L^2({\R^N}, \rho_w)$ the operator $\sigma(h)\circ B$ 
%where $\sigma:L^2({\R^N}, \rho_w) \to L_0(L^2({\R^N}, \rho_w_w))$ and $B$ is given by \eqref{B} 
is in the space $L_2(Q^\frac{1}{2}(U), H)$.  We know by Example \ref{bounded is HS} that it suffices to prove that $\sigma(h)\circ B:L^2({\R^N}) \to L^2({\R^N}, \rho_w)$ is bounded for any $h\in L^2({\R^N}, \rho_w)$.  To this end, for any $u\in L^2({\R^N})$, we have by definition
\begin{align*}
\|\sigma(h)\circ B(u)\|^2_{L^2({\R^N}, \rho_w)} &\leq \|\sigma(h)\|^2_{L_0(L^2({\R^N}, \rho_w))} \|B(u)\|^2_{L^2({\R^N}, \rho_w)} \\
&= \|\sigma(h)\|^2_{L_0(L^2({\R^N}, \rho_w))}\int_{\R^N} | B(u)(x)|^2\rho_w(x)dx\\
&=   \|\sigma(h)\|^2_{L_0(L^2({\R^N}, \rho_w))}\int_{\R^N} \left(\int_{{\R^N}} \varphi(x-y)u(y)dy\right)^2\rho_w(x)dx \\
%& \leq \|\sigma(h)\|^2_{L_0(L^2({\R^N}, \rho_w))}\|\rho_w\|_\infty \|u \star \varphi \|_{L^2({\R^N})}^2\\
%&\leq \|\sigma(h)\|^2_{L_0(L^2({\R^N}, \rho_w))} \|\varphi\|_{L^1({\R^N})}\|\rho_w\|_\infty  \int_{\R^N} \int_{{\R^N}} u^2(y)\varphi(x-y)dydx \\
& \leq \|\sigma(h)\|^2_{L_0(L^2({\R^N}, \rho_w))}\|\rho_w\|_\infty  \|\varphi\|^2_{L^1({\R^N})} \|u\|^2_{L^2({\R^N})}<\infty,
\end{align*}
since $\rho_w$ is bounded (in either case) and $\varphi\in L^1({\R^N})$.

\vspace{0.2cm}
\noindent (H4):  To show (H4), we first want $\mathbf{F}: L^2({\R^N}, \rho_w) \to L^2({\R^N}, \rho_w)$ to be globally Lipschitz.  To this end, for any $g, h\in L^2({\R^N}, \rho_w)$, we see that in either case
\begin{align*}
&\|\mathbf{F}(g) - \mathbf{F}(h)\|_{L^2({\R^N}, \rho_w)}^2 = \int_{{\R^N}}|\mathbf{F}(g) - \mathbf{F}(h)|^2(x)\rho_w(x)dx\\
&\qquad \leq \int_{\R^N}\left(\int_{\R^N}|w(x,y)|\left|G(g(y)) - G(h(y))\right|dy \right)^2\rho_w(x)dx\\
&\qquad \leq C^2_G \int_{\R^N}\left(\int_{\R^N}|w(x,y)|\left|g(y) - h(y)\right|dy \right)^2\rho_w(x)dx,
%&\qquad \leq C^2_G \sup_{x\in{\R^N}}\|w(x, \cdot)\|_{L^1({\R^N})}\int_{\R^N}\left(\int_{\R^N}\left|g(y) - h(y)\right|^2|w(x,y)|dy \right) \rho_w(x)dx\\
\end{align*}
where we have used the Lipschitz property of $G$.
Now in case $(i)$ it clearly follows from the Cauchy-Schwartz inequality that
\begin{align*}
\|\mathbf{F}(g) - \mathbf{F}(h)\|_{L^2({\R^N})}^2  \leq C^2_G \left(\int_{\R^N}\int_{\R^N} |w(x, y)|^2 dx dy \right) \left\|g - h\right\|_{L^2({\R^N})},
\end{align*}
so that by condition (\textbf{C1}), $\mathbf{F}$ is indeed Lipschitz. 

In case (ii), by  Cauchy-Schwartz and the specific property of $\rho_w$ given by (\textbf{C1'}), we see that
\begin{align*}
&\|\mathbf{F}(g) - \mathbf{F}(h)\|_{L^2({\R^N}, \rho_w)}^2  \\
&\qquad \leq C^2_G \sup_{x\in{\R^N}}\|w(x, \cdot)\|_{L^1({\R^N})} \int_{\R^N} \left|g(y) - h(y)\right|^2\left(\int_{\R^N} |w(x,y)| \rho_w(x)dx\right)dy\\
& \qquad\leq C^2_G\Lambda_w \sup_{x\in{\R^N}}\|w(x, \cdot)\|_{L^1({\R^N})} \|g - h\|_{L^2({\R^N}, \rho_w)}^2,
\end{align*}
so that again $\mathbf{F}$ is Lipschitz. 

The final step is to show that $\sigma(\cdot)\circ B: H\to L_2(Q^\frac{1}{2}(U), H)$ with $U=L^2(\R^N)$ and $H=L^2(\R^N, \rho_w)$ is Lipschitz.  Again, by Example \ref{bounded is HS} we have for any $g, h\in H$
\begin{align*}
\left\|\sigma(g)\circ B - \sigma(h)\circ B \right\|^2_{L_2(Q^\frac{1}{2}(U), H)} &\leq \mathrm{Tr}(Q)\|\sigma(g)\circ B - \sigma(h)\circ B \|^2_{L_0(U, H)}\\
&\leq \mathrm{Tr}(Q)\|\sigma(g) - \sigma(h)\|^2_{L_0(H)}\|B\|^2_{L_0(U, H)}\\
&\leq C_\sigma^2 \mathrm{Tr}(Q)\|B\|^2_{L_0(U, H)}\|g - h\|^2_{H},
\end{align*}
where $\|B\|_{L_0(U, H)}$ is finite since $\rho_w$ is bounded (in either case).
\end{proof}

\vspace{0.5cm}
As mentioned we now present two important cases where the conditions (\textbf{C1'}) and (\textbf{C2'}) are satisfied.

\vspace{0.5cm}
\noindent\textbf{Example 1: $|w|$ defines a compact integral operator.}
Suppose that 
\begin{itemize}
\item given $\varepsilon > 0$, there exists $\delta>0$ and $R>0$ such that for all $\theta\in{\R^N}$ with $|\theta|<\delta$
\begin{itemize}
\item[(i)] for almost all $x\in{\R^N}$,
\[
\int_{{\R^N}\backslash B(0, R)} |w(x, y)| dy < \varepsilon, \qquad \int_{{\R^N}}|w(x, y+\theta)-w(x, y)|dy < \varepsilon,
\]
\item[(ii)] for almost all $y\in{\R^N}$,
\[
\int_{{\R^N}\backslash B(0, R)} |w(x, y)| dx < \varepsilon, \qquad \int_{{\R^N}}|w(x+\theta, y)-w(x, y)|dx < \varepsilon,
\]
\end{itemize}
where $B(0, R)$ denotes the ball of radius $R$ in ${\R^N}$ centered at the origin;
\item There exists a bounded subset $\Omega \subset {\R^N}$ of positive measure such that
\[
\inf_{y \in\Omega}\int_\Omega |w(x, y)|dx >0, \qquad \mathrm{or} \qquad \inf_{x \in\Omega}\int_\Omega |w(x, y)|dy >0;
\]
\item $w$ satisfies (\textbf{C2'}) and moreover
\[
\forall y\in{\R^N}\ (x\mapsto w(x, y))\in L^1({\R^N}), \quad \mathrm{and}\quad \sup_{y\in{\R^N}}\|w(\cdot, y)\|_{L^1({\R^N})} < \infty.
\]
\end{itemize}
We claim that these assumptions are sufficient for (\textbf{C1'}) so that we can apply Theorem \ref{NF SEE:e+u} in this case.
Indeed, let $\mathbb{X}$ be the Banach space of functions in $L^1({\R^N})\cap L^\infty({\R^N})$ equipped with the norm $\|\cdot\|_\mathbb{X} = \max\{\|\cdot\|_{L^1({\R^N})}, \|\cdot\|_{L^\infty({\R^N})}\}$.
Thanks to the last point above, we can well-define the map $J:\mathbb{X} \to \mathbb{X}$ by
\[
Jh(y) = \int_{\R^N} |w(x, y)|h(x)dx, \quad h\in \mathbb{X}.
\] 
%$J:\mathbb{X} \to \mathbb{X}$ is well-defined, since for any $h\in \mathbb{X}$
%\[
%\|J(h)\|_{L^1({\R^N})}  
%\leq \sup_{x\in{\R^N}}\|w(x,\cdot)\|_{L^1({\R^N})}\|h\|_{L^1({\R^N})} <\infty
%\]
%and
%\[
%\|J(h)\|_{L^\infty({\R^N})}  
%\leq  \sup_{y\in{\R^N}}\|w(\cdot, y)\|_{L^1({\R^N})}\|h\|_{L^\infty({\R^N})} <\infty
%\]
Moreover, it follows from \cite[Corollary 5.1]{Eveson} that the first condition we have here imposed on $w$ is in fact necessary and sufficient for both the operators $J:L^1({\R^N})\to L^1({\R^N})$ and $J:L^\infty({\R^N})\to L^\infty({\R^N})$ to be compact.  We therefore clearly also have that the condition is necessary and sufficient for the operator $J:\mathbb{X}\to\mathbb{X}$ to be compact.

Note now that the space $\mathbb{K}$ of positive functions in $\mathbb{X}$ is a cone in $\mathbb{X}$ such that $J(\mathbb{K}) \subset \mathbb{K}$, and that the cone is \textit{reproducing} (i.e. $\mathbb{X} = \{f - g: f, g \in \mathbb{K}\}$). If we can show that $r(J)$ is strictly positive, we can thus finally apply the Krein-Rutman Theorem (see for example \cite[Theorem 1.1]{Du}) to see that $r(J)$ is an eigenvalue with corresponding non-zero eigenvector $\rho\in\mathbb{K}$.

To show that $r(J)>0$, suppose first of all that there exists a bounded $\Omega \subset {\R^N}$ of positive measure such that $\inf_{y \in\Omega}\int_\Omega |w(x, y)|dx >0$. Define $h=1$ on $\Omega$, $0$ elsewhere, so that $\|h\|_{\mathbb{X}} = \max\{1, |\Omega|\}$.
Then, trivially,
\begin{align*}
\|Jh\|_{\mathbb{X}} \geq \sup_{y\in{\R^N}} \int_\Omega |w(x, y)|dx  & \geq \inf_{y\in\Omega}\int_{\Omega} |w(x, y)|dx =: m >0,
\end{align*}
by assumption.  Replacing $h$ by $\tilde{h} = h/ \max\{1, |\Omega|\}$ yields $\|\tilde{h}\|_\mathbb{X}=1$ and 
\[
\|J\tilde{h}\|_\mathbb{X} \geq m/  \max\{1, |\Omega|\}.
\]
Thus $\|J\| \geq m/ \max\{1, |\Omega|\}$.
Similarly
%\begin{align*}
%\|J^2\tilde{h}\|_{\mathbb{X}} &\geq %\frac{1}{ \max\{1, |\Omega|\}}\sup_{y\in{\R^N}}\int_{\R^N} |w(x_1, y)| Jh(x_1) dx_1\\
%%&= 
%\frac{1}{ \max\{1, |\Omega|\}}\sup_{y\in{\R^N}}\int_{\R^N} |w(x_1, y)| \left(\int_{\Omega}|w(x_2, x_1)|dx_2\right) dx_1\\
%&\geq \frac{1}{ \max\{1, |\Omega|\}} \int_{\R^N} |w(x_1, y)| \left(\int_{\Omega}|w(x_2, x_1)|dx_2\right) dx_1, \qquad \forall y\in {\R^N}\\
%%&\geq \frac{1}{ \max\{1, |\Omega|\}} \int_\Omega |w(x_1, y)| \left(\int_{\Omega}|w(x_2, x_1)|dx_2\right) dx_1, \qquad \forall y\in {\R^N}\\
%&\geq \frac{1}{ \max\{1, |\Omega|\}} \inf_{x_1\in\Omega}\left(\int_{\Omega}|w(x_2, x_1)|dx_2\right) \int_\Omega |w(x_1, y)| dx_1, \qquad \forall y\in {\R^N}.
%\end{align*}
\begin{align*}
\|J^2{h}\|_{\mathbb{X}} &\geq \sup_{y\in{\R^N}}\int_{\R^N} |w(x_1, y)| \left(\int_{\Omega}|w(x_2, x_1)|dx_2\right) dx_1\\
&\geq \int_{\R^N} |w(x_1, y)| \left(\int_{\Omega}|w(x_2, x_1)|dx_2\right) dx_1, \qquad \forall y\in {\R^N}\\
%&\geq \int_\Omega |w(x_1, y)| \left(\int_{\Omega}|w(x_2, x_1)|dx_2\right) dx_1, \qquad \forall y\in {\R^N}\\
&\geq  \inf_{x_1\in\Omega}\left(\int_{\Omega}|w(x_2, x_1)|dx_2\right) \int_\Omega |w(x_1, y)| dx_1, \qquad \forall y\in {\R^N}.
\end{align*}
Therefore 
\begin{align*}
\|J^2{h}\|_{\mathbb{X}} 
 \geq m^2,
\end{align*}
so that $\|J^2\|\geq m^2/\max\{1, |\Omega|\}$.  In fact we have $\|J^k\|\geq m^k/\max\{1, |\Omega|\}$ for all $k\geq1$, so that, by the spectral radius formula, $r(J)\geq m>0$.  The case where $\inf_{x \in\Omega}\int_\Omega |w(x, y)|dy >0$ holds instead is proved similarly, by instead taking $h = 1/|\Omega|$ on $\Omega$ ($0$ elsewhere) and working with the $L^1({\R^N})$ norm of $Jh$ in place of the $L^\infty({\R^N})$ norm.

We have thus found a non-negative, non-zero function $\rho = \rho_w \in L^1({\R^N})\cap L^\infty({\R^N})$ such that
\[
\int_{\R^N} |w(x, y)|\rho_w(x)dx = r(J)\rho_w(y), \qquad \forall y\in{\R^N},
\]
so that (\textbf{C1'}) is satisfied.

\vspace{0.5cm}
\noindent\textbf{Example 2: Homogeneous case.}  Suppose that
\begin{itemize}
\item $w$ is homogeneous i.e $w(x,y) = w(x-y)$ for all $x, y\in {\R^N}$;
\item $w\in L^1({\R^N})$ and is continuous;
\item $\int_{\R^N}|x|^{2N}|w(x)|dx <\infty$.% is absolutely integrable. %$x\mapsto |x|^{2N}w(x)$ is absolutely integrable.
\end{itemize}
\label{page:example2}
These conditions are satisfied for many typical choices of the neural field kernel in the literature (e.g. the ``Mexican hat'' kernel \cite{bressloff-cowan-etal:01, faye-chossat-etal:11, owen-laing-etal:07,veltz-faugeras:10}).  However, it is clear that we are not in the case of the previous example, since for any $R>0$
\[
\sup_{x\in{\R^N}}\int_{{\R^N}\backslash B(0, R)}|w(x-y)|dy = \|w\|_{L^1({\R^N})},
\]
which is not uniformly small.
We thus again show that (\textbf{C1'}) is satisfied in this case so that (since (\textbf{C2'}) is trivially satisfied)  Theorem \ref{NF SEE:e+u} yields the existence of a unique $L^2({\R^N}, \rho_w)$-valued solution to \eqref{NF SEE}.

In order to do this, we use the Fourier transform.  Let $v = |w|$, so that $v$ is continuous and in $L^1({\R^N})$.  Let $\mathfrak{F}v$ be the Fourier transform of $v$ i.e.
\[
\mathfrak{F}v (\xi) := \int_{\R^N} e^{-2\pi i x.\xi}v(x)dx,  \quad \xi\in {\R^N}.
\]
Therefore $\mathfrak{F}v$ is continuous and bounded by
\[
\sup_{\xi\in{\R^N}}|\mathfrak{F}v(\xi)| \leq \|v\|_{L^1({\R^N})} = \|w\|_{L^1({\R^N})}.
\]
Now let $\Lambda_w = \|w\|_{L^1({\R^N})} +1$, and $z(x) := e^{-|x|^2/2}$, $x\in {\R^N}$, so that $z$ is in the Schwartz space of smooth rapidly decreasing functions, which we denote by $\cS({\R^N})$.  Then define
\[
\hat{\rho}(\xi) := \frac{\mathfrak{F}z(\xi)}{\Lambda_w - \mathfrak{F}v(\xi)}.
\]

We note that the denominator is continuous and strictly bounded away from $0$ (indeed by construction $\Lambda_w- \mathfrak{F}v(\xi) \geq 1$ for all $\xi\in{\R^N}$).  Thus $\hat{\rho}$ is continuous, bounded and in $L^1({\R^N})$ (since $\mathfrak{F}z\in \cS({\R^N})$ by the standard stability result for the Fourier transform on $\cS({\R^N})$).

We now claim that $\mathfrak{F}^{-1}\hat{\rho}(x) \in L^1({\R^N})$, where the map $\mathfrak{F}^{-1}$ is defined by

\[
\mathfrak{F}^{-1}g(x) := \int_{\R^N} e^{2\pi i x.\xi}g(\xi)d\xi, \quad g\in L^1({\R^N}).
\]
Indeed, we note that for any $k\in\{1, \dots, N\}$,
\[
\partial^{2N}_{k}\mathfrak{F}v(\xi) =  (-2\pi i)^{2N} \int_{\R^N}  e^{-2\pi i x.\xi}x^{2N}_kv(x)dx,
\]
which is well-defined and bounded thanks to our assumption on the integrability of $x\mapsto |x|^{2N}|w(x)|$.  Since $\mathfrak{F} z$ is rapidly decreasing, we can thus see that the function $\hat{\rho}(\xi)$ is $2N$ times differentiable with respect to every component and $\partial^{2N}_{k}\hat{\rho}(\xi)$ is absolutely integrable for every $k\in\{1, \dots N\}$.  Finally, since $\mathfrak{F}^{-1}(\partial^{2N}_k\hat{\rho})(x) = (2\pi i)^{2N} x_k^{2N}\mathfrak{F}^{-1}\hat{\rho}(x)$ for each $k\in\{1, \dots, N\}$, we have that
\[
|\mathfrak{F}^{-1}\hat{\rho}(x)| \leq \frac{\sum_{k=1}^N |\mathfrak{F}^{-1}(\partial^{2N}_k\hat{\rho})(x)|}{(2\pi)^{2N} \sum_{k=1}^Nx_k^{2N}} \leq \frac{N^{N-1}\sum_{k=1}^N \|\partial^{2N}_k\hat{\rho}\|_{L^1(\R^N)}}{(2\pi)^{2N} |x|^{2N}},
\]
for all $x\in\R^N$.  Thus there exists a constant $K$ such that $|\mathfrak{F}^{-1}\hat{\rho}(x)| \leq K/|x|^{2N}$.  Moreover, since we also have the trivial bound
\[
|\mathfrak{F}^{-1}\hat{\rho}(x)| \leq \|\hat{\rho}\|_{L^1({\R^N})},
\]
for all $x\in {\R^N}$, it follows that $|\mathfrak{F}^{-1}\hat{\rho}(x)| \leq K/(1+|x|^{2N})$, by adjusting the constant $K$.  Since this is integrable over $\R^N$, the claim is proved.

%\[
%\frac{d}{d\xi_k}\mathfrak{F}v(\xi) =  -2\pi i \int_{\R^N}  e^{-2\pi i x.\xi}x_kv(x)dx,\quad \frac{d^2}{d\xi_k^2}\mathfrak{F}v(\xi) =  -4\pi^2 \int_{\R^N}  e^{-2\pi i x.\xi}x_k^2v(x)dx,
%\]
%which are both well-defined and bounded thanks to our assumption on the integrability of $x\mapsto xw(x)$ and $x\mapsto x^2w(x)$.  Since $\mathfrak{F} z$ is rapidly decreasing, we can thus see that the function $\hat{\rho}(\xi)$ is twice differentiable and such that $\hat{\rho}''(\xi)$ is absolutely integrable.  Finally, since $\mathfrak{F}^{-1}(\hat{\rho}'')(x) = -4\pi^2x^2 \mathfrak{F}^{-1}\hat{\rho}(x)$ we have that
%\[
%|\mathfrak{F}^{-1}\hat{\rho}(x)| \leq \|\hat{\rho}''\|_{L^1({\R^N})}/4\pi^2x^2, \quad\mathrm{and}\quad |\mathfrak{F}^{-1}\hat{\rho}(x)| \leq \|\hat{\rho}\|_{L^1({\R^N})},
%\]
%for all $x\in {\R^N}$, from which it follows that $|\mathfrak{F}^{-1}\hat{\rho}(x)| \leq K/(1+x^2)$ for some constant $K$, which is integrable.

%Indeed, it is easy to see that 
%\[
%\int_{\R^N}|\rho(x)|dx \leq \int_{{\R^N}} z(x)dx <\infty.
%\]
%Since the Fourier Inversion Theorem {\color{red}Folland, Theorem 7.6} holds with $\mathfrak{F}^{-1}$ defined this way for continuous absolutely integrable functions, we thus have that $\hat{\rho}(\xi) = \mathfrak{F}\rho(\xi)$, for all $\xi\in{\R^N}$.

Now, by the classical Fourier Inversion Theorem (which is applicable since $\hat{\rho}$ and $\mathfrak{F}^{-1}\hat{\rho}$ are both in $L^1({\R^N})$), we thus have that
\[
\mathfrak{F}\left(\mathfrak{F}^{-1}\hat{\rho}\right)(\xi) = \hat{\rho}(\xi),
\]
for all $\xi\in {\R^N}$. 

By setting $\rho(x) = \mathfrak{F}^{-1}\hat{\rho}(x)$, we see that
\[
\Lambda_w\mathfrak{F}\rho(\xi) - \mathfrak{F}\rho(\xi)\mathfrak{F}v(\xi) :=\mathfrak{F}z(\xi).
\] 
We may finally again apply the inverse Fourier transform $\mathfrak{F}^{-1}$ to both sides, so that by the Inversion Theorem again (along with the standard convolution formula) it holds that
\[
\Lambda_w\rho(y) - \int_{{\R^N}}v(x-y)\rho(x)dx =e^{-\frac{|y|^2}{2}}, \qquad y\in {\R^N}.
\]
It then follows that
\[
\int_{{\R^N}}|w(x-y)|\rho(x)dx \leq \Lambda_w\rho(y), \qquad y\in {\R^N},
\]
as claimed.

\begin{rem}[Large Deviation Principle]
\label{LDP}
The main focus of \cite{RiedlerKuehn} was a large deviation principle for the stochastic neural field equation \eqref{NF SEE} with small noise, but in a less general situation than we consider here.  In particular, the authors only considered the neural field equation driven by a simple additive noise, white in both space and time.

We would therefore like to remark that in our more general case, and under much weaker conditions than those imposed in \cite{RiedlerKuehn}, an LDP result still holds and can be quoted from the literature. Indeed, such a result is presented in \cite[Theorem 7.1]{Peszat}.  The main conditions required for the application of this result have essentially already been checked above (global Lipschitz properties of $\mathbf{F}$ and $\sigma(\cdot)\circ B$), and it thus remains to check conditions (E.1) -- (E.4) as they appear in \cite{Peszat}.  In fact these are trivialities, since the strongly continuous contraction semigroup $S(t)$ is generated by the identity in our case.
\end{rem}

\subsection{Color of the noise in stochastic neural field equation \eqref{NF SEE}}
\label{color}

It is important to understand the properties of the noise term in the neural field equation \eqref{NF SEE} which we now know has a solution in some sense, and in particular why we have have chosen the particular form \eqref{B} for the `coefficient' $B$ (although it is really an operator).  Recall that by definition
\[
B(u)(x) = \int_{{\R^N}} \varphi(x-y)u(y)dy, \qquad x\in{\R^N},\ u\in L^2({\R^N}),
\]
for some $\varphi\in L^1({\R^N})$.  The first point is that we have deliberately made the definition for $\varphi\in L^1({\R^N})$ so that it is possible to (at least formally) take $\varphi$ as the Dirac delta function.  More rigorously we can take smooth approximations in $L^1({\R^N})$ that integrate to $1$.  Anyhow, the operator $B$ is then simply the identity and Theorem \ref{NF SEE:e+u} still holds.

However, when $\varphi$ is not the Dirac function, we claim that the noise term $BdW(t)$ in \eqref{NF SEE} is in fact spatially correlated.  This is important in applications. For example in \cite{bressloff-webber-2012} it is actually a spatially correlated noise that is added to the deterministic neural field equation.

To see the spatial correlation, consider $B$ as a map from $L^2({\R^N}) \to L^2({\R^N}, \rho)$ for some bounded $\rho\in L^1({\R^N})$.  As noted above, $B$ is then bounded, and hence the process $(U(t))_{t\geq0}$ defined by
\[
U(t) = \int_0^tBdW(s), \qquad t\geq0,
\] 
is a well-defined $L^2({\R^N}, \rho)$-valued process.  Moreover, by Theorem 5.2 of \cite{D-Z}, $(U(t))_{t\geq0}$ is Gaussian with mean zero and
\[
\mathrm{Cov} \left(U(t)U(s)\right) = s\wedge t BQB^*, \quad s, t \geq 0,
\]
where $B^*:L^2({\R^N}, \rho)\to L^2({\R^N})$ is the adjoint of $B$. In other words, for all $g, h \in L^2({\R^N}, \rho)$, $s, t\geq0$, we have that
\[
\mE{\langle g, U(s) \rangle_{L^2({\R^N}, \rho)}\langle h, U(t) \rangle_{L^2({\R^N}, \rho)}} = s\wedge t\langle BQB^*g, h\rangle_{L^2({\R^N}, \rho)}.
\]
That is, for any $g, h \in L^2({\R^N}, \rho)$
\begin{align}
\label{covariance 1}
&\int_{\R^N}\int_{\R^N}\mE{U(s, x) U(t, y)}g(x)h(y)\rho(x)\rho(y)dxdy = s\wedge t\left\langle QB^*h, B^*g\right\rangle_{L^2({\R^N})}\nonumber\\
&\qquad\qquad\qquad =  s\wedge t\int_{\R^N} QB^*g(z)B^*h(z)dz\nonumber\\
&.\qquad\qquad\qquad =  s\wedge t\int_{\R^N} Q^{1/2}B^*g(z)Q^{1/2}B^*h(z)dz.
\end{align}
Now, by definition, for $u\in L^2({\R^N})$ and $f\in L^2({\R^N}, \rho)$
\begin{align*}
\int_{\R^N} u(y) B^*(f)(y)dy &= \int_{\R^N}  B(u)(x)f(x)\rho(x)dx \\
& =  \int_{\R^N} u(y)\int_{\R^N} \varphi(x-y)f(x)\rho(x)dxdy
\end{align*}
so that $B^*(f)(y) = \int_{\R^N}\varphi(x-y)f(x)\rho(x)dx$.  Using this in \eqref{covariance 1}, we see that
\begin{align*}
&\int_{\R^N}\int_{\R^N}\mE{U(s, x) U(t, y)}g(x)h(y)\rho(x)\rho(y)dxdy\\
&\qquad =  s\wedge t\int_{\R^N}  Q^\frac{1}{2}\left(\int_{\R^N}\varphi(x-z)g(x)\rho(x)dx\right)  Q^\frac{1}{2}\left(\int_{{\R^N}}\varphi(y-z)h(y)\rho(y)dy\right)dz\\
&\qquad =  s\wedge t\int_{\R^N}  \left(\int_{\R^N} Q^\frac{1}{2}\varphi(x-z)g(x)\rho(x)dx\right)  \left(\int_{{\R^N}}Q^\frac{1}{2}\varphi(y-z)h(y)\rho(y)dy\right)dz,
%&\qquad = \int_{{\R^N}}\int_{{\R^N}}\left(s\wedge t\int_{{\R^N}}\varphi(x-z)\varphi(y-z)dz\right)h(x)g(y)\rho(x)\rho(y)dydx,
\end{align*}
for all $g, h \in L^2({\R^N}, \rho)$, since $Q$ is a linear operator and is self-adjoint.
We can then conclude that 
\begin{equation}
\label{bresloff covariance 1}
\mE{U(s, x) U(t, y)} = s\wedge t\int_{\R^N} Q^\frac{1}{2}\varphi(x-z)Q^\frac{1}{2}\varphi(y-z)dz = (s\wedge t)c(x-y),
\end{equation}
where $c(x) = Q^\frac{1}{2}\varphi*Q^\frac{1}{2}\tilde{\varphi}(x)$ and $\tilde{\varphi}(x)=\varphi(-x)$. Hence $(U(t))_{t\geq0}$ is white in time but stationary and colored in space with covariance function $(s\wedge t)c(x)$.  This is exactly the rigorous interpretation of the noise described in \cite{bressloff-webber-2012}, when interpreting a solution to the stochastic neural field equation as a process taking values in $L^2({\R^N}, \rho_w)$.\\
%\olivier{This is very interesting. The noise spatial correlation depends only upon the choice of the operator $Q$ and the function $\varphi$. Are there any "good" ways of choosing them for neuroscience applications? We would lose stationarity if we would drop the hypothesis that $\varphi$ is translation invariant. } 

%%%%%%%%%%%%%%%%%%%%%%%%%%%%%%%%
%
%%%%%%%%%%%%%%%%%%%%%%%%%%%%%%%%

\section{Stochastic neural fields as Gaussian random fields}
\label{random fields}

In this section we take an alternative approach, and try to give sense to a solution to the stochastic neural field equation \eqref{bressloff} as a random field, using Walsh's theory of integration.  

This approach generally takes as its starting point a deterministic PDE, and then attempts include a term which is random in both space and time. With this in mind, consider first the well studied deterministic neural field equation
\begin{equation}
\label{NF det}
\partial_tY(t, x) =  - Y(t, x) + \int_{\R^N} w(x,y)G(Y(t, y))dy, \quad x\in\R^N,\ t\geq0.
\end{equation}
Under some conditions on the neural field kernel $w$ (boundedness, condition (\textbf{C2'}) above and $L^1$-Lipschitz continuity), this equation has a unique solution $(t, x)\mapsto Y(t, x)$ that is bounded and continuous in $x$ and continuously differentiable in $t$, whenever $x\mapsto Y(0, x)$ is bounded and continuous (\cite{Potthast}).

The idea then is to directly add a noise term to this equation, and try and give sense to all the necessary objects in order to be able to define what we mean by a solution.  Indeed, consider the following stochastic version of \eqref{NF det},
\begin{equation}
\label{NF RF}
\partial_tY(t, x) =  - Y(t, x) + \int_{\R^N} w(x, y)G(Y(t, y))dy + \sigma(Y(t, x))\dot{W}(t, x)
\end{equation}
where $\dot{W}$ is a ``space-time white noise''.  The definition of $\dot{W}$ will be made precise below, but informally we may think of the object $\dot{W}(t, x)$ as the random \textit{distribution} which, when integrated against a test function $h\in L^2(\R^+\times{\R^N})$
\[
\dot{W}(h):= \int_0^\infty\int_{\R^N} h(t, x)\dot{W}(t, x)dtdx, \qquad h\in L^2(\R^+\times{\R^N}),
\]
yields a zero-mean Gaussian random field  $(\dot{W}(h))_{h\in L^2(\R^+\times{\R^N})}$  with covariance 
\[
\mE{\dot{W}(g)\dot{W}(h)} = \int_0^\infty\int_{\R^N} g(t, x) h(t, x) dxdt, \qquad g, h\in L^2(\R^+\times{\R^N}).
\] 

The point is that with this interpretation of space-time white noise, since equation \eqref{NF RF} specifies no regularity in the spatial direction (the map $x\mapsto Y(t,x)$ is simply assumed to be Lebesgue measurable so that the integral makes sense), it is clear that any solution will be \textit{distribution}-valued in the spatial direction, which is rather unsatisfactory.  Indeed, consider the extremely simple linear case when $G\equiv 0$ and $\sigma \equiv 1$, so that \eqref{NF RF} reads
\begin{equation}
\label{NF linear}
\partial_tY(t, x) =  - Y(t, x) + \dot{W}(t, x).
\end{equation}
Formally, the solution to this equation is given by 
\[
Y(t, x) = e^{-t}Y(0, x) + \int_0^t e^{-(t-s)}\dot{W}(s, x) ds, \qquad t\geq0, x\in{\R^N},
\]
and since the integral is only over time it is clear (at least formally) that $x\mapsto Y(t, x)$ is a distribution for all $t\geq0$.  This differs significantly from the usual SPDE situation, when one would typically have an equation such as \eqref{NF linear} where a second order differential operator in space is applied to the first term on the right-hand side (leading to the much studied \textit{stochastic heat equation}).  In such a case, the semigroup generated by the second order differential operator can be enough to smooth the space-time white noise in the spatial direction, leading to solutions that are continuous in both space and time (at least when the spatial dimension is $1$ -- see for example \cite[Chapter 3]{Pardoux} or \cite[Chapter 3]{Walsh}).

Of course one can develop a theory of distribution-valued processes (as is done in \cite[Chapter 4]{Walsh}) to interpret solutions of \eqref{NF RF} in the obvious way:  one says that the random field $(Y(t, x))_{t\geq0, x\in{\R^N}}$ is a (weak) solution to \eqref{NF RF} if for all $\phi\in C_0^\infty({\R^N})$ it holds that
\begin{align*}
\int_{{\R^N}}\phi(x) Y(t, x)dx &= e^{-t}\int_{\R^N}\phi(x) Y(0, x)dx \\
& \quad + \int_0^t\int_{\R^N} e^{-(t-s)}\phi(x)\int_{\R^N} w(x, y)G(Y(s, y))dydxds\\
& \quad + \int_0^t\int_{\R^N} e^{-(t-s)}\phi(x)\sigma(Y(s,x))\dot{W}(s, x)dxds,
\end{align*}
for all $t\geq 0$.  Here all the integrals can be well-defined, which makes sense intuitively if we think of $\dot{W}(t, x)$ as a distribution.  In fact it is more common to write  $\int_0^t\int_{\R^N} e^{-(t-s)}\phi(x)W(dsdx)$ for the stochastic integral term, once it has been rigorously defined. 

%However, we argue that it is not worth developing this theory here for two reasons: i) distribution-valued solutions are of little interest physically, and ii) we already have an interpretation of a solution to the stochastic neural field equation driven by noise that is white in both space and time as an $L^2({\R^N}, \rho_w)$-valued process, as developed in the previous section (this corresponds to taking the Dirac function as the kernel in the operator $B$ given in \eqref{B}). 

However, we argue that it is not worth developing this theory here, since distribution valued solutions are of little interest physically.
It is for this reason that we instead look for other types of random noise to add to the deterministic equation \eqref{NF det} that will produce solutions that are real-valued random fields, and are at least H\"older continuous in both space and time.  In the theory of SPDEs, when the spatial dimension is $2$ or more, the problem of an equation driven by space-time white noise having no real-valued solution is a well-known and much studied one (again see for example \cite[Chapter 3]{Pardoux} or \cite[Chapter 3]{Walsh} for a discussion of this).  To get around the problem, a common approach (\cite{Dalang-Frangos, Sanz-Sole, Sanz-Sole2}) is to consider random noises that are smoother than white noise, namely a Gaussian noise that is white in time but has a smooth spatial covariance.  Such random noise is known as either spatially colored or spatially homogeneous white-noise.  One can then formulate conditions on the covariance function to ensure that real-valued H\"older continuous solutions to the specific SPDE exist.

It should also be mentioned, as remarked in \cite{Dalang-Frangos}, that in trying to model physical situations, there is some evidence that white-noise smoothed in the spatial direction is more natural, since spatial correlations are typically of a much larger order of magnitude than time correlations. 

In the stochastic neural field case, since we have no second order differential operator, our solution will only ever be as smooth as the noise itself.  We therefore look  to add a noise term to \eqref{NF det} that is at least H\"older continuous in the spatial direction, and then proceed to look for solutions to the resulting equation in the sense of Walsh. 

The section is structured as follows.  First we briefly introduce Walsh's theory of stochastic integration with respect to martingale measures, for which the classical reference is \cite{Walsh} although we instead follow closely the more recent explanation given by D. Khoshnevisan in \cite{DalangKhosh-course}.  This theory will be needed to well-define the stochastic integral in our definition of a solution to the neural field equation.  We then introduce the spatially smoothed space-time white noise that we will consider, before finally applying the theory  to analyze solutions of the neural field equation driven by this spatially smoothed noise under certain conditions.

\subsection{Integration with respect to the white noise process}
\label{Walsh}

Consider the centered Gaussian random field\footnote{Recall that a collection of random variables $X = \{X(\theta)\}_{\theta\in \Theta}$ indexed by a set $\Theta$ is a Gaussian random field on $\Theta$ if $(X(\theta_1), \dots, X(\theta_k))$ is a $k$-dimensional Gaussian random vector for every $\theta_1, \dots, \theta_k\in \Theta$.  It is characterized by its mean and covariance functions.}
 \[
 \dot{W} := (\dot{W}(A))_{A\in\mathcal{B}(\R^+\times {\R^N})}
 \]
 indexed by $A\in \mathcal{B}(\R^+\times{\R^N})$ (where $\R^{+} := [0, \infty)$) with covariance function
\begin{equation}
\label{white noise}
\mE{\dot{W}(A)\dot{W}(B)} =|A\cap B|, \quad A, B, \in  \mathcal{B}(\R^+\times{\R^N}),
\end{equation}
where $|A\cap B|$ denotes the Lebesgue measure of $A\cap B$. We say that $\dot{W}$ is a \textit{white noise} on $\R^+\times{\R^N}$.  We then define the \textit{white noise process} $W:=(W_t(A))_{t\geq0, A\in \mathcal{B}({\R^N})}$ by 
\begin{equation}
\label{white noise process}
W_t(A):= \dot{W}([0, t]\times A),
\end{equation}
for all $t\geq 0$, and we suppose that $(W_t(A))_{t\geq0}$ is adapted to the filtration $(\cF_t)_{t\geq0}$ for all $A\in \mathcal{B}({\R^N})$.

We would like to build up a theory of stochastic integration with respect to this process.  With this in mind, one may hope that $W_t$ is a signed measure on ${\R^N}$ for all $t>0$.  However, for all such $t$ it holds that
\[
\lim_{n\to\infty}\sum_{j=0}^{2^n - 1}\left|W_t\left(\left[\frac{j-1}{2^n}, \frac{j}{2^n}\right]\right)\right| = \infty
\]
almost surely (see Exercise 3.16 \cite[Chapter 1]{DalangKhosh-course}), so that $W_t$ is in fact not $\sigma$-finite with any positive probability.

On the other hand, we can prove that
\[
\mathbb{P}\left(W_t\left(\bigcup_{i=1}^\infty A_i \right) = \sum_{i=1}^\infty W_t(A_i) \right) = 1
\]
for all disjoint sets $A_1, A_2, \dots \in \mathcal{B}({\R^N})$ and all $t\geq0$, and that the sum is convergent in $L^2(\Omega, \cF, \mathbb{P})$ (see Lemma 5.3 \cite[Chapter 1]{DalangKhosh-course}).  In this sense $W_t$ is instead an $L^2(\Omega, \cF, \mathbb{P})$-\textit{valued} (random) measure i.e. $W_t: \mathcal{B}({\R^N}) \to L^2(\Omega, \cF, \mathbb{P})$.

Moreover, it is straightforward to show that for all $A\in \cB({\R^N})$, $(W_t(A))_{t\geq0}$ is a centered martingale (with respect to $(\cF_t)_{t\geq0}$).  In summary we have that the white noise process $W :=(W_t(A))_{t\geq0, A\in \cB({\R^N})}$ is such that 
\begin{itemize}
\item[{\rm(i)}] $W_0(A) = 0$ almost surely, for all $A\in \cB({\R^N})$;
\item[{\rm(ii)}] for all $t>0$, $W_t$ is a $\sigma$-finite $L^2(\Omega, \cF, \bbP)$-valued signed measure; 
\item[{\rm(iii)}] for all $A\in\cB({\R^N})$, $(W_t(A))_{t\geq0}$ is a centered martingale with respect to the filtration $(\cF_t)_{t\geq0}$.
\end{itemize}
In general, a family of random variables indexed by $t\geq0$ and $A\in \cB({\R^N})$ satisfying (i)-(iii) is defined to be a \textit{martingale measure} (with respect to $(\cF_t)_{t\geq0}$).

One can in fact build stochastic integrals with respect to general martingale measures under a condition known as \textit{worthiness} (\cite[Chapter 2]{Walsh}).  However, for our needs and to keep things simple, we concentrate on this construction for the white noise process (which turns out to be worthy).

Indeed, starting with \textit{elementary} functions $f: \R^+\times {\R^N} \times \Omega \to \R$ of the form
\begin{equation}
\label{elementary functions}
f(t, x, \omega) = X(\omega)\Ind_{(a, b]}(t)\Ind_A(x)
\end{equation}
where $a, b \in \R^+$, $X:\Omega\to \R$ is bounded and $\cF_a$-measurable, and $A\in \cB({\R^N})$, we first define the stochastic integral process of $f$ with respect to $W$ as
\[
(f\cdot W)_t(B)(\omega) := X(\omega)\left[W_{t\wedge b}(A\cap B) - W_{t\wedge a}(A\cap B)\right](\omega), \quad t\geq0, B\in \cB({\R^N}).
\]
It is important to note that $f\cdot W$ is itself a new martingale measure (exactly as It\^o integrals with respect to martingales are martingales).  As usual we then build up the definition to accommodate integrands of linear combinations of elementary functions.  
We call such functions \textit{simple} functions, and denote the set of all simple functions by $\mathfrak{S}$.
We will also say that a function $(t,x,\omega)\mapsto f(t,x,\omega)$ is \textit{predictable} if it is measurable with respect to the $\sigma$-algebra generated by $\mathfrak{S}$, which we denote by $\mathcal{P}$.  In other words, $\mathcal{P}$ is the smallest $\sigma$-algebra on $\R^+\times{\R^N}\times\Omega$ such that all simple functions are measurable.

As with the construction of the It\^o integral, to go beyond linear combinations of elementary functions the quadratic variation process plays a role.  Indeed, we define
\[
Q_W([0, t), A, B) := < W_\cdot(A), W_\cdot(B)>_t, \qquad \forall t\geq 0,\ A, B \in \cB({\R^N}),
\]
where $< \cdot, \cdot >_t$ is the standard cross-variation process between two martingales. The point is that this process defines a $\sigma$-finite measure on $\R^+\times {\R^N} \times{\R^N}$, since by \eqref{white noise}
\begin{align*}
Q_W([0, t), A, B)  &= \E[W_t(A)W_t(B)]= \E[\dot{W}([0,t]\times A) \dot{W}([0, t]\times B)]\\
&= |([0,t]\times A)\cap ([0,t]\times B)|\\
&= t |A\cap B|,
\end{align*}
for all $t\geq0$, $A, B \in \cB({\R^N})$, so that
\begin{equation}
\label{Q measure}
Q_W(dtdxdy) = dt\delta_0(x-y)dxdy,
\end{equation}
where $\delta_0$ is the Dirac delta function.

\begin{rem}
It should be noted that when building stochastic integrals with respect to a general martingale measure, one needs to impose extra conditions in order to ensure that the quadratic variation process can be associated with a $\sigma$-finite measure.  
It is at this point that the afore mentioned worthiness property is needed.
\end{rem}

To proceed, we now fix a finite horizon $T>0$ and define the norm
\begin{equation}
\label{MM norm}
\| f\|^2_W :=  \E\left[\int_0^T\int_{{\R^N}}\int_{{\R^N}} |f(t, x)f(t, y)|Q_W(dtdxdy)\right] = \E\left[\int_0^T\int_{{\R^N}} |f(t, x)|^2dtdx\right],
\end{equation}
for any predictable function $f$.  Then let $\mathfrak{P}_W$ be the set of all predictable functions $f$ for which $\|f\|_{W} <\infty$.  

 The following proposition and theorem complete the construction of the stochastic integral with respect to $W$.  For proofs we refer to \cite[Proposition 2.3 and Theorem 2.5]{Walsh} respectively (which are written for the general case of a worthy martingale measure).

\begin{prop}
The space $\mathfrak{P}_W$ equipped with the norm $\|\cdot\|_W$ is a complete Banach space.  Moreover, the space of simple functions $\mathfrak{S}$ is dense in $\mathfrak{P}_W$.
\end{prop}

\begin{thm}
\label{thm: int}
For all $f\in \mathfrak{P}_W$, $f\cdot W$ can be well-defined as the $L^2(\Omega, \cF, \mathbb{P})$-limit of martingale measures $f_n\cdot W$, for an approximating sequence $\{f_n\}_{n\geq1}\subset \mathfrak{S}$ in the norm $\|\cdot\|_W$. Moreover $f\cdot W$ is a martingale measure such that for all $t\in (0, T]$ and $A, B\in \cB({\R^N})$,
\[
<(f\cdot W)_\cdot(A), (f\cdot W)_\cdot(B)>_t = \int_0^t\int_A\int_Bf(s,x)f(s, y)Q_W(dtdxdy),
\]
and 
\begin{equation}
\label{L2 burkholder}
\E[(f\cdot W)^2_t(A)] \leq \|f\|^2_W.
\end{equation}
\end{thm}

The $L^p$-version of \eqref{L2 burkholder} is known as Burkh\"older's inequality and will be useful:
\begin{thm}[Burkh\"older's inequality]
\label{thm: burkholder}
For all $p\geq2$ there exists a constant $c_p$ such that for all $f\in\mathfrak{P}_W$, $t\in (0, T]$ and $A\in\cB({\R^N})$,
\[
\E[|(f\cdot W)_t(A)|^p] \leq c_p\E\left[\left(\int_0^T\int_{{\R^N}}|f(t, x)|^2dtdx\right)^\frac{p}{2}\right].
\]
\end{thm}

We now adopt the more standard notation and set 
\begin{equation}\label{notation}
(f\cdot W)_t(A) =: \int_0^t\int_AfdW =   \int_0^t\int_Af(s, x)W(dsdx),
\end{equation}
for all $f\in\mathfrak{P}_W$, $t\geq0$ and $A\in\cB({\R^N})$.

\subsection{Spatially smoothed space-time white noise}
\label{smoothed noise}

Let $W = (W_t(A))_{t\geq0, A\in\cB({\R^N})}$ be a white-noise process adapted to $(\cF_t)_{t\geq0}$ as defined in the previous section.  For $\varphi \in L^2({\R^N})$, we can well-define the (Gaussian) random field $(W^\varphi(t, x))_{t\geq0, x\in {\R^N}}$ for any $T>0$ by 
\begin{equation}
\label{smoothed GN}
W^\varphi(t, x) := \int_0^t \int_{\R^N} \varphi(x-y)W(dsdy).
\end{equation}
To see this one just needs to check that $\varphi(x - \cdot)\in \mathfrak{P}_W$ for every $x$, where, as above, $ \mathfrak{P}_W$ is the set of all predictable functions $f$ for which $\|f\|_{W} <\infty$.  The function $\varphi(x - \cdot)$ is clearly predictable for each $x$ (since it is non-random) and for every $T>0$
\begin{align*}
\|\varphi(x - \cdot)\|_W^2 &= \E\left[\int_0^T\int_{{\R^N}}|\varphi(x-z)|^2dtdz \right]\\
&= T\|\varphi\|_{L^2({\R^N})}^2 < \infty,
\end{align*}
so that the integral in \eqref{smoothed GN} is indeed well-defined in the sense of the above construction.  Moreover, $(W^\varphi(t, x))_{t\geq0}$ is a centered martingale for each $x\in{\R^N}$ with respect to $(\cF_t)_{t\geq0}$ (by the properties of martingale measures) and has spatial covariance
\begin{align*}
\E[W^\varphi(t, x)W^\varphi(t, y)] &= \E\left[\int_0^t \int_{\R^N} \varphi(x-u)\,W(dsdu) \times \int_0^t \int_{\R^N} \varphi(y-v)\,W(dsdv)\right]\\
&= \langle  (\varphi(x-\cdot) \cdot W)_\cdot ({\R^N}),\,  (\varphi(y-\cdot) \cdot W)_\cdot ({\R^N})\rangle_t,
\end{align*}
which by equation \eqref{notation}, Theorem \ref{thm: int} and equation \eqref{Q measure} is equal to 
\[
 t\int_{\R^N}\varphi(x-z)\varphi(y-z)dz = t\varphi \star \tilde{\varphi} (x-y), \qquad \forall t\geq0, x,y\in {\R^N},
\]
where $\star$ denotes the convolution operator as usual, and $\tilde{\varphi}(x)=\varphi(-x)$.  In this sense the noise is again spatially correlated. 

The regularity in time of this process is the same as that of a Brownian path:
\begin{lem}
For any $x\in{\R^N}$, the path $t\mapsto W^\varphi(t, x)$ has an $\eta$-H\"older continuous modification for any $\eta\in(0, 1/2)$.
\end{lem}

\begin{proof}
For $x\in {\R^N}$, $s,t\geq 0$ with $s\leq t$ and any $p\geq 2$ we have by Burkh\"older's inequality (Theorem \ref{thm: burkholder} above) that
\begin{align*}
\mE{\left|W^\varphi(t, x) - W^\varphi(s, x)\right|^p} \leq c_p\|\varphi\|_{L^2({\R^N})}^2(t-s)^\frac{p}{2}.
\end{align*}
The result follows from the standard Kolmogorov continuity theorem (see for example Theorem 4.3 of \cite[Chapter 1]{DalangKhosh-course}).
\end{proof}

More importantly, if we impose some (very weak) regularity on $\varphi$ then $W^\varphi$ inherits some spatial regularity:

\begin{lem}
\label{lem:phi regularity}
Suppose that there exists a constant $C_\varphi$ such that
\begin{equation}
\label{phi regularity}
\|\varphi - {\boldsymbol \tau}_{z}(\varphi)\|_{L^2({\R^N})} \leq C_\varphi |z|^\alpha, \qquad \forall z\in{\R^N},
\end{equation}
for some $\alpha\in(0, 1]$, where ${\boldsymbol \tau}_{z}$ indicates the shift by $z$ operator (so that ${\boldsymbol \tau}_z(\varphi)(y):=\varphi(y+z)$ for all $y, z\in{\R^N}$).  Then for all $t\geq0$, the map $x\mapsto W^\varphi(t, x)$ has an $\eta$-H\"older continuous modification, for any $\eta\in (0, \alpha)$.
\end{lem}

\begin{proof}
For $x, \tilde{x} \in {\R^N}$, $t\geq 0$, and any $p\geq 2$ we have (again by Burkh\"older's inequality) that
\begin{align*}
\mE{\left| W^\varphi(t, x) - W^\varphi(t, \tilde{x})\right|^p} 
&\leq  t^\frac{p}{2}c_p\left( \int_{\R^N} |\varphi(x-y) - \varphi(\tilde{x}-y)|^2dy\right)^\frac{p}{2}\\
&=  t^\frac{p}{2}c_p\left( \int_{\R^N} |\varphi(y) - \varphi(y+\tilde{x} - x)|^2dy\right)^\frac{p}{2}\\
&\leq t^\frac{p}{2}c_pC_\varphi^p|x -\tilde{x}|^{p\alpha}.
\end{align*}
The result follows by Kolmogorov's continuity theorem.
\end{proof}

\begin{rem}
The condition \eqref{phi regularity} with $\alpha =1$ is true if and only if the function $\varphi$ is in the Sobolev space $W^{1, 2}({\R^N})$ (\cite[Proposition 9.3]{Brezis}). 

When $\alpha <1$ the set of functions $\varphi\in L^2({\R^N})$ which satisfy \eqref{phi regularity} defines a Banach space denoted by $N^{\alpha, 2}({\R^N})$ which is known as the Nikolskii space.  This space is closely related to the more familiar fractional Sobolev space $W^{\alpha, 2}({\R^N})$ though they are not identical.  We refer to \cite{Jaques} for a detailed study of such spaces and their relationships.  An example of when \eqref{phi regularity} holds with $\alpha =1/2$ is found by taking $\varphi$ to be an indicator function. It is in this way we see that \eqref{phi regularity} is a rather weak condition.
\end{rem}

%{\color{red}Question: can we let $\varphi \to \delta$? i.e. dose the stochastic integral still make sense for $\varphi = \delta$?}

\subsection{The stochastic neural field equation driven by spatially smoothed space-time white noise}
\label{NF smoothed GN}
We now have everything in place to define and study the solution to the stochastic neural field equation driven by a spatially smoothed space-time white noise.  Indeed, consider the equation
\begin{equation}
\label{NF GN}
\partial_tY(t, x) =  - Y(t, x) + \int_{\R^N} w(x,y)G(Y(t, y))dy + \sigma(Y(t, x))\frac{\partial}{\partial t}W^\varphi(t, x), 
\end{equation}
with initial condition $Y(0, x) = Y_0(x)$ for $x\in{\R^N}$ and $t\geq0$, where 
\begin{itemize}
\item $(W^\varphi(t, x))_{t\geq0, x\in {\R^N}}$ is the spatially smoothed space-time white noise (adapted to $(\cF_t)_{t\geq0}$) defined by \eqref{smoothed GN} for some $\varphi\in L^2({\R^N})$;
\item $G:\R \to \R$ is the nonlinear gain function, assumed to bounded and globally Lipschitz i.e such that there exists a constant $C_G$ with $\sup_{a\in\R}|G(a)| \leq C_G$  and 
\[
|G(a) - G(b)| \leq C_G|a- b|, \qquad \forall a, b \in \R,
\]
as above;
\item $\sigma:\R \to \R$ is globally Lipschitz i.e. there exists a constant $C_\sigma$ such that
\[
|\sigma(a) - \sigma(b)| \leq C_\sigma|a-b|, \ \mathrm{and} \ \ |\sigma(a)| \leq C_\sigma(1 + |a|), \qquad \forall a, b\in\R.
\]
\end{itemize}

Although the above equation is not well-defined ($\frac{\partial}{\partial t}W^\varphi(t, x)$ does not exist), we will interpret a solution to \eqref{NF GN} in the following way.

\begin{defin}
\label{def:mild solution GN}
By a solution to \eqref{NF GN} we will mean a real-valued predictable (i.e. $\mathcal{P}$-measurable) adapted random field $(Y(t, x))_{t\geq0, x\in {\R^N}}$ such that
\begin{align}
\label{mild solution GN}
Y(t, x) &=  e^{-t}Y_0(x) + \int_0^te^{-(t-s)}\int_{\R^N} w(x, y)G(Y(s, y))dyds \nonumber\\
&\qquad\qquad \qquad + \int_0^t\int_{{\R^N}}e^{-(t-s)}\sigma(Y(s, x))\varphi(x-y)W(dsdy),
\end{align}
almost surely for all $t\geq0$ and $x\in {\R^N}$, where the stochastic integral term is understood in the sense described in Section \ref{Walsh}.
\end{defin}

Once again we are interested in the conditions on the neural field kernel $w$ that allow us to prove the existence of solutions in this new sense.  Recall that in Section \ref{EE} we either required conditions (\textbf{C1}) and (\textbf{C2}) or  (\textbf{C1'}) and (\textbf{C2'}) to be satisfied.  The difficulty was to keep everything well-behaved in the Hilbert space $L^2({\R^N})$ (or $L^2({\R^N}, \rho)$).  However, when looking for solutions in the sense of random fields $(Y(t, x))_{t\geq0, x\in {\R^N}}$ such that \eqref{mild solution GN} is satisfied, such restrictions are no longer needed, principally because we no longer have to concern ourselves with the behavior in space at infinity.  Indeed, in this section we simply work with the condition  (\textbf{C2'}) i.e. that 
\[
\forall x\in{\R^N}\ (y\mapsto w(x, y))\in L^1({\R^N}), \quad \mathrm{and}\quad \sup_{x\in{\R^N}}\|w(x,\cdot)\|_{L^1({\R^N})} \leq C_w,
\]
for some constant $C_w$.

\begin{thm}
\label{thm: e+u GN}
Suppose that the map $x\mapsto Y_0(x)$ is Borel measurable almost surely, $Y_0(x)$ is $\cF_0$-measurable for all $x\in\R^N$, 
and that
%\begin{equation}
%\label{initial uniform bound}
\[
\sup_{x\in{\R^N}}\mE{|Y_0(x)|^2} < \infty.
\]
%\end{equation}
Suppose moreover that the neural field kernel $w$ satisfies condition {\rm ({\bf C2'})}.  Then there exists an almost surely unique
%in $L^2(\Omega, \cF, \mathbb{P})$
predictable random field $(Y(t, x))_{t\geq0, x\in {\R^N}}$ which is a solution to \eqref{NF GN} in the sense of Definition \ref{def:mild solution GN} such that
\begin{equation}
\label{thm: e+u GN L2 bound}
\sup_{t\in[0, T], x\in{\R^N}}\mE{|Y(t, x)|^2} < \infty,
\end{equation}
for any  $T>0$.
\end{thm}

\begin{proof} The proof proceeds in a classical way, but where we are careful to interpret all stochastic integrals as described in Section \ref{Walsh}.

\vspace{0.2cm}
\noindent\textit{Uniqueness:} Suppose that $(Y(t, x))_{t\geq 0, x\in{\R^N}}$ and $(Z(t, x))_{t\geq 0, x\in{\R^N}}$ are both solutions to \eqref{NF GN} in the sense of Definition \ref{def:mild solution GN}. Let $D(t, x) = Y(t, x) - Z(t, x)$ for $x\in {\R^N}$ and $t\geq0$. Then we have
\begin{align*}
D(t, x) &= \int_0^te^{-(t-s)}\int_{\R^N} w(x, y)[G(Y(s, y)) - G(Z(s, y))]dyds \\
& \qquad + \int_0^t\int_{{\R^N}}e^{-(t-s)}[\sigma(Y(s, x)) - \sigma(Z(s, x))]\varphi(x-y)W(dsdy).
\end{align*}
Therefore
\begin{align*}
\mE{|D(t, x)|^2} &\leq 2\mE{\left(\int_0^te^{-(t-s)}\int_{\R^N} |w(x, y)||G(Y(s, y)) - G(Z(s, y))|dyds\right)^2} \\
&  + 2\mE{\left(\int_0^t\int_{{\R^N}}e^{-(t-s)}[\sigma(Y(s, x)) - \sigma(Z(s, x))]\varphi(x-y)W(dsdy)\right)^2}\\
&\leq 2t\int_0^te^{-2(t-s)}\mE{\left(\int_{\R^N} |w(x, y)||G(Y(s, y)) - G(Z(s, y))|dy\right)^2}ds \\
& \quad + 2\int_0^t\int_{{\R^N}}e^{-2(t-s)}\mE{|\sigma(Y(s, x)) - \sigma(Z(s, x))|^2}|\varphi(x-y)|^2dsdy,
\end{align*}
where we have used Cauchy-Schwarz and the $L^2$-version of Burkh\"older's inequality \eqref{L2 burkholder}. Thus, using the Lipschitz property of $\sigma$ and $G$,
\begin{align*}
\mE{|D(t, x)|^2} &\leq 2tC_G^2\int_0^te^{-2(t-s)}\mE{\left(\int_{\R^N} |w(x, y)||D(s, y)|dy\right)^2}ds \\
& \qquad + 2C_\sigma^2\|\varphi\|^2_{L^2({\R^N})}\int_0^te^{-2(t-s)}\mE{|D(s, x)|^2}ds.
\end{align*}
By the Cauchy-Schwarz inequality once again
\begin{align*}
\mE{|D(t, x)|^2} &\leq 2tC_G^2\|w(x, \cdot)\|_{L^1({\R^N})} \int_0^te^{-2(t-s)}\int_{\R^N} |w(x, y)|\,\mE{|D(s, y)|^2}dyds \\
& \qquad + 2C_\sigma^2\|\varphi\|^2_{L^2({\R^N})}\int_0^te^{-2(t-s)}\mE{|D(s, x)|^2}ds.
\end{align*}

Let $H(s) := \sup_{x\in{\R^N}}\E [| D(s, x)|^2]$, which is finite since we are assuming  $Y$ and $Z$ satisfy \eqref{thm: e+u GN L2 bound}.  Writing $K=2\max\{C^2_\sigma, C^2_G\}$, we have
\begin{align*}
\mE{|D(t, x)|^2} &\leq K\left[tC_w^2 +  \|\varphi\|_{L^2({\R^N})}^2\right]\int_0^t e^{-2(t-s)}H(s)ds \\
\Rightarrow H(t) & \leq K\left[tC_w^2 +  \|\varphi\|_{L^2({\R^N})}^2\right]\int_0^t H(s)ds. 
\end{align*}
An application of Gronwall's lemma then yields $\sup_{s\leq t}H(s) = 0$ for all $t\geq 0$.  Hence $Y(t, x) = Z(t, x)$ almost surely for all $t\geq0,  x\in {\R^N}$. 

\vspace{0.2cm}
\noindent\textit{Existence:}
Let $Y_0(t, x) = Y_0(x)$.  Then define iteratively for $n\in\mathbb{N}_0$, $t\geq0$, $x\in{\R^N}$,
\begin{align}
\label{Picard}
Y_{n+1}(t, x) &:= e^{-t} Y_0(x) + \int_0^te^{-(t-s)}\int_{\R^N} w(x, y)G(Y_n(s, y))dyds \nonumber\\
& \quad + \int_0^t\int_{{\R^N}}e^{-(t-s)}\sigma(Y_n(s, x))\varphi(x-y)W(dsdy).
\end{align}
We first check that the stochastic integral is well-defined, under the assumption that
\begin{equation}
\label{tocheck}
\sup_{t\in [0, T], x\in{\R^N}}\E(|Y_n(t, x)|^2) <\infty,
\end{equation}
for any $T>0$, which we know is true for $n=0$ by assumption, and we show by induction is also true for each integer $n\geq 1$ below.  To this end for any $T>0$
\begin{align*}
&\E\left[\int_0^T\int_{{\R^N}}e^{-2(t-s)}\sigma^2(Y_n(s, x))\varphi^2(x-y)dsdy\right]  \\
&\qquad \leq 2C_\sigma^2\|\varphi\|^2_{L^2({\R^N})}\int_0^T(1 + \E\left[|Y_n(s, x)|^2\right])ds\\
&\qquad \leq 2C_\sigma^2\|\varphi\|^2_{L^2({\R^N})}T\left[ 1 + \sup_{t\in [0, T], x\in{\R^N}}\mE{|Y_n(t, x)|^2}\right]<\infty.
\end{align*}
This shows that the integrand in the stochastic integral is in the space $\mathfrak{P}_W$ (for all $T>0$), which in turn implies that the stochastic integral in the sense of Walsh is indeed well-defined (by Theorem \ref{thm: int}).

To be rigorous, we must moreover check that the deterministic integral in \eqref{Picard} is well-defined.  When $n=0$, this follows from the fact that $x\mapsto Y_0(x)$ is Borel measurable almost surely.  For $n\geq 1$, we use the fact that the stochastic convolution 
\[
(t, x) \mapsto \int_0^t\int_{{\R^N}}e^{-(t-s)}\sigma(Y_{n-1}(s, x))\varphi(x-y)W(dsdy)
\]
is predictable if $(t, x)\mapsto Y_{n-1}(t, x)$ is predictable (this follows from the construction in Section \ref{Walsh} and is explicitly stated in \cite[Section 2.1]{Conus}).  Hence, for fixed $t$ the map $x\mapsto Y_n(t, x)$ is Borel measurable almost surely, which in turn allows us to well-define the deterministic integral in \eqref{Picard}.

Now define $D_n(t, x) := Y_{n+1}(t, x) - Y_{n}(t, x)$ for $n\in\mathbb{N}_0$, $t\geq0$ and $x\in{\R^N}$.  Then exactly as in the uniqueness calculation we have
\begin{align*}
\mE{|D_{n}(t, x)|^2} &\leq  2tC^2_GC_w\int_0^te^{-2(t-s)}\int_{\R^N}|w(x,y)|\mE{\left|D_{n-1}(s, y)\right|^2} dyds\\
& \quad + 2C^2_\sigma\|\varphi\|_{L^2({\R^N})}^2 \int_0^t\mE{\left|D_{n-1}(s,x)\right|^2}e^{-2(t-s)}ds .
\end{align*}
This implies that by setting $H_n(s) = \sup_{x\in {\R^N}} \mE{\left|D_{n}(s, x)\right|^2}$,
\begin{align}
\label{Hn}
H_n(t) %&\leq K\left[t\|w\|_{L^1(dx)}^2 +  \|\varphi\|_{L^2(dx)}^2\right] \int_0^t H_{n-1}(s)ds\nonumber\\
&\leq K^n\left[tC_w^2 +  \|\varphi\|_{L^2({\R^N})}^2\right]^n \int_0^t\int_0^{t_1}\dots\int_0^{t_{n-1}}H_{0}(t_n)dt_n\dots dt_1,
\end{align}
for all $n\in\mathbb{N}_0$ and $t\geq0$.
Now, similarly, we can find a constant $C_t$ such that
\begin{align*}
&\mE{\left|D_0(s, x)\right|^2} \leq C_t\left( 1 + \sup_{x\in {\R^N}} \mE{|Y_0(x)|^2} \right),
\end{align*}
for any $x\in {\R^N}$ and $s\in[0, t]$, so that for $s\in[0, t]$,
\begin{align*}
H_0(s) = \sup_{x\in{\R^N}} \mE{\left|D_{0}(s, x) \right|^2} \leq C_t \left( 1 + \sup_{x\in {\R^N}} \mE{|Y_0(x)|^2} \right),
\end{align*}
Using this in \eqref{Hn} we see that,
\[
H_n(t) \leq {C}_tK^{n}\left[tC_w^2 +  \|\varphi\|_{L^2({\R^N})}^2\right]^{n}  \left( 1 + \sup_{x\in {\R^N}} \mE{|Y_0(x)|^2} \right)\frac{t^n}{n!},
\]
for all $t\geq0$.  This is sufficient to see that \eqref{tocheck} holds uniformly in $n$.  
%We thus have that
%\[
%\sup_{t\in[0, T]}\sup_{x\in{\R^N}} \E\left|Y_{n+1}(t, x) - Y_n(t,x)\right|^2\to 0
%\]
%as $n\to\infty$.  
By completeness, for each $t\geq 0$ and $x\in{\R^N}$ there exists $Y(t, x)\in L^2(\Omega, \cF, \mathbb{P})$ such that $Y(t, x)$ is the limit in $L^2(\Omega, \cF, \mathbb{P})$ of the sequence of  square-integrable random variables $(Y_n(t, x))_{n\geq1}$.   Moreover, the convergence is uniform on $[0, T]\times{\R^N}$, i.e.
\[
\sup_{t\in[0, T], x\in{\R^N}} \E\left|Y_n(t, x) - Y(t,x)\right|^2\to 0.
\]
From this we can see that \eqref{thm: e+u GN L2 bound} is satisfied for the random field $(Y(t, x))_{t\geq0, x\in{\R^N}}$.  It remains to show that $(Y(t, x))_{t\geq0, x\in{\R^N}}$ satisfies \eqref{mild solution GN} almost surely.   By the above uniform convergence, we have that
\begin{align*}
\mE{\left|\int_0^t\int_{{\R^N}}e^{-(t-s)} \left[\sigma(Y_n(s,x)) -  \sigma(Y(s,x)) \right]\varphi(x-y)W(dsdy)\right|^2} \to 0,
\end{align*}
and
\begin{align*}
\mE{\left|\int_0^te^{-(t-s)} \int_{{\R^N}}w(x, y)\left[G(Y_n(s,y)) -  G(Y(s,y)) \right]dsdy\right|^2} \to 0,
\end{align*}
uniformly for all $t\geq0$ and $x\in {\R^N}$.
Thus taking the limit as $n\to \infty$ in \eqref{Picard} (in the $L^2(\Omega, \cF, \mathbb{P})$ sense) proves that $(Y(t, x))_{t\geq0, x\in{\R^N}}$ does indeed satisfy \eqref{mild solution GN} almost surely.  The fact that it is predictable and adapted follows from the construction.
\end{proof}

In a very similar way, one can also prove that the solution remains $L^p$-bounded whenever the initial condition is $L^p$-bounded for any $p>2$.  Moreover, this also allows us to conclude that the solution has time continuous paths for all $x\in{\R^N}$.

\begin{thm}
\label{thm: Lp boundedness}
Suppose that we are in the situation of Theorem \ref{thm: e+u GN}, but in addition we have that $\sup_{x\in{\R^N}}\mE{|Y_0(x)|^p} <\infty$ for some $p>2$.  
Then the solution $(Y(t,x))_{t\geq0, x\in{\R^N}}$ to \eqref{NF GN} in the sense of Definition \ref{def:mild solution GN} is $L^p$-bounded on $[0,T]\times{\R^N}$ for any $T$ i.e.
\[
\sup_{t\in[0, T],x\in{\R^N}}\mE{\left|Y(t, x)\right|^p} <\infty,
\]
and the map $t\mapsto Y(t, x)$ has a continuous version for all $x\in{\R^N}$.
 If the initial condition has finite $p$-moments for all $p>2$, then $t\mapsto Y(t, x)$ has an $\eta$-H\"older continuous version, for any $\eta\in(0, 1/2)$ and any $x\in{\R^N}$.
\end{thm}

\begin{proof}
The proof of the first part of this result uses similar techniques as in the proof of Theorem \ref{thm: e+u GN} in order to bound $\mE{|Y(t, x)|^p}$ uniformly in $t\in[0, T]$ and $x\in{\R^N}$.  In particular, we use the form of $Y(t, x)$ given by \eqref{mild solution GN}, Burkh\"older's inequality (see Theorem \ref{thm: burkholder}), H\"older's inequality and Gronwall's lemma, as well as the conditions imposed on $w$, $\sigma$, $G$ and $\varphi$.

For the time continuity, we again use similar techniques to achieve the bound
\[
\mE{\left|Y(t, x) - Y(s, x)\right|^p} \leq C_T^{( p)}\left(1 + \sup_{r\in[0, T],y\in{\R^N}}\mE{|Y(r, y)|^p}\right)(t-s)^\frac{p}{2},
\]
for all $s,t\in[0, T]$ with $s\leq t$ and $x\in{\R^N}$, for some constant $C_T^{(p)}$.  The results then follow from Kolmogorov's continuity theorem once again.
\end{proof}

\subsubsection{Spatial regularity of solution}

As mentioned in the introduction to this section, the spatial regularity of the solution $(Y(t, x))_{t\geq0, x\in{\R^N}}$ to \eqref{NF GN} is of interest.  In particular we would like to find conditions under which it is at least continuous in space.  As we saw in Lemma \ref{lem:phi regularity}, under the weak condition on $\varphi$ given by \eqref{phi regularity}, we have that the spatially smoothed space-time white noise is continuous in space.  We here show that under this assumption together with a H\"older continuity type condition on the neural field kernel $w$, the solution $(Y(t, x))_{t\geq0, x\in{\R^N}}$ inherits the spatial regularity of the the driving noise.

The condition we introduce on $w$ is the following:
\begin{equation}\label{eq:C3'}
\tag{\textbf{C3'}}
\exists K_w \geq 0 \ \ \mathrm{s.t.}\ \ \|w(x, \cdot) - w(\tilde{x}, \cdot)\|_{L^1({\R^N})} \leq L_w |x-\tilde{x}|^\alpha, \quad \forall x, \tilde{x} \in{\R^N},
\end{equation}
for some $\alpha\in(0, 1]$.

\begin{rem}
This condition is certainly satisfied for all typical choices of neural field kernel $w$.  In particular, any smooth rapidly decaying function will satisfy $({\bf C3'})$.
\end{rem}

\begin{thm}[Regularity]
Suppose that we are in the situation of Theorem \ref{thm: e+u GN} and 
\[
\sup_{x\in{\R^N}}\mE{|Y_0(x)|^p} < \infty
 \]
for all $p\geq2$.  Suppose moreover that there exists $\alpha\in(0, 1]$ such that
\begin{itemize}
\item $w$ satisfies {\rm ({\bf C3'})};
\item $\varphi$ satisfies \eqref{phi regularity} i.e.
\[
\|\varphi - {\boldsymbol \tau}_z(\varphi)\|_{L^2({\R^N})} \leq C_\varphi |z|^\alpha, \qquad \forall z\in{\R^N},
\]
where ${\boldsymbol \tau}_z$ indicates the shift by $z\in{\R^N}$ operator;
\item $x\mapsto Y_0(x)$ is $\alpha$-H\"older continuous.
\end{itemize}
Then $(Y(t, x))_{t\geq0, x\in{\R^N}}$ has a modification such that $(t, x)\mapsto Y(t,x)$ is $(\eta_1, \eta_2)$-H\"older continuous, for any $\eta_1 \in (0,1/2)$ and $\eta_2 \in (0,\alpha)$.
\end{thm}

\begin{proof}
Let $(Y(t, x))_{t\geq0, x\in {\R^N}}$ be the mild solution to \eqref{NF GN}, which exists and is unique by Theorem \ref{thm: e+u GN}.  The stated regularity in time is given in Theorem \ref{thm: Lp boundedness}.  It thus remains to prove the regularity in space.

Let $t\geq0$, $x\in{\R^N}$.  Then by \eqref{mild solution GN} 
\begin{align}
\label{reg thm: mild solution}
Y(t, x) &=  e^{-t}Y_0(x) + I_1(t, x) + I_2(t, x),
\end{align}
for all $t\geq0$ and $x\in {\R^N}$, where $I_1(t, x) =  \int_0^te^{-(t-s)}\int_{\R^N} w(x, y)G(Y(s, y))dyds$ and $I_2(t, x) = \int_0^t\int_{{\R^N}}e^{-(t-s)}\sigma(Y(s, x))\varphi(x-y)W(dsdy)$.

Now let $p\geq2$.  The aim is to estimate $\mE{|Y(t, x) - Y(t, \tilde{x})|^p}$ for $x, \tilde{x} \in {\R^N}$ and then to use Kolmogorov's theorem to get the stated spatial regularity. To this end, we have that
\begin{align}
\label{reg in x 1}
&\mE{\left|I_1(t, x) - I_1(t, \tilde{x}) \right|^p}\nonumber \\
&\leq \mE{\left( \int_0^t\int_{\R^N} |w(x, y) - w(\tilde{x}, y)||G(Y(s, y))|dyds \right)^p}\nonumber\\
&\leq C_G^pt^p \|w(x, \cdot) - w(\tilde{x}, \cdot)\|_{L^1({\R^N})}^p\nonumber\\
&\leq C_G^pt^pK^p_w |x-\tilde{x}|^{p\alpha},
\end{align}
where we have used (\textbf{C3'}).  Moreover, by H\"older's and Burkh\"older's inequalities once again, we see that
\begin{align*}
&\mE{\left|I_2(t, x) - I_2(t, \tilde{x}) \right|^p}\nonumber \\
&\quad\leq 2^{p-1}\mE{\left|\int_0^t\int_{{\R^N}}e^{-(t-s)}\left[\sigma(Y(s, x))-  \sigma(Y(s, \tilde{x}))\right] \varphi(x - y)W(dyds)\right|^p} \nonumber\\
&\qquad + 2^{p-1}\mE{\left|\int_0^t\int_{{\R^N}}e^{-(t-s)}\sigma(Y(s, \tilde{x}))[\varphi(x- y) - \varphi(\tilde{x}-y)]W(dyds)\right|^p} \nonumber\\
&\quad\leq 2^{p-1}c_p\mE{\left(\int_0^t\int_{{\R^N}}|\sigma(Y(s, x))-  \sigma(Y(s, \tilde{x}))|^2\left| \varphi(x - y)\right|^2dyds\right)^\frac{p}{2}} \nonumber\\
&\qquad + 2^{p-1}c_p\mE{\left(\int_0^t\int_{{\R^N}}\left|\sigma(Y(s, \tilde{x}))\right|^2|\varphi(x - y) - \varphi(\tilde{x}-y)|^2dyds\right)^\frac{p}{2}}\nonumber,
\end{align*}
for all $x, \tilde{x} \in {\R^N}$ and $p\geq2$.  Thus
\begin{align}
\label{reg in x 2}
&\mE{\left|I_2(t, x) - I_2(t, \tilde{x}) \right|^p}\nonumber \\
&\leq 2^{p-1}c_pC_\sigma^pt^{\frac{p}{2} -1}\|\varphi\|_{L^2({\R^N})}^p\int_0^t\mE{|Y(s, x)-  Y(s, \tilde{x})|^p}ds \nonumber\\
&\qquad + 2^{2(p-1)}c_pC_\sigma^pt^{\frac{p}{2} }\|\varphi - \boldsymbol{\tau}_{\tilde{x}-x}(\varphi)\|_{L^2({\R^N})}^p\left(1 + \sup_{s\in[0, T], y\in{\R^N}}\mE{\left|Y(s, y)\right|^p}\right),
\end{align}
where we note that the right-hand side is finite thanks to Theorem \ref{thm: Lp boundedness}.
Returning to \eqref{reg thm: mild solution} and using estimates \eqref{reg in x 1} and \eqref{reg in x 2} we see that there exists a constant $C_T^{( p)}$ (depending on $T, p, C_G, K_w, C_\sigma, C_\varphi, \|\varphi\|_{L^2({\R^N})}$ as well as $\sup_{s\in[0, T], y\in{\R^N}}\mE{|Y(s, y)|^p}$), such that
\begin{align*}
&\mE{\left|Y(t, x) - Y(t, \tilde{x}) \right|^p}\nonumber \\
&\quad\leq C_T^{( p)}\left[\mE{\left|Y_0(x) - Y_0(\tilde{x}) \right|^p} + |x - \tilde{x}|^{p\alpha} + \int_0^t\mE{|Y(s, x)-  Y(s, \tilde{x})|^p}ds \right]\\
&\quad\leq C_T^{( p)}\left[|x - \tilde{x}|^{p\alpha} + \int_0^t\mE{|Y(s, x)-  Y(s, \tilde{x})|^p}ds \right],
\end{align*}
where the last line follows from our assumptions on $Y_0$ and by adjusting the constant $ C_T^{( p)}$.
This bound holds for all $t\geq 0$, $x, \tilde{x} \in{\R^N}$ and $p\geq2$.  The proof is then completed using Gronwall's inequality,  and Kolmogorov's continuity theorem once again.
\end{proof}

%\olivier{This is very nice. We need to think about what happens when $x$ and $y$ live in $\R^n$, $n \geq 2$. It should be OK. It would also be interesting to compare the two approaches from the neural field viewpoint. It is surprising that the first approach seems to be easier to develop for a translation invariant connectivity function $w$. We need to discuss the various options: $w$ translation invariant or not, spatial domain infinite or bounded, of dimension 1 or higher.}\\

%%%%%%%%%%%%%%%%%%%%%%%%
%Comparison
%%%%%%%%%%%%%%%%%%%%%%%%

\section{Comparison of the two approaches}
\label{sec: comparison}
The purpose of this section is to compare the two different approaches taken in Sections \ref{EE} and \ref{random fields} above to give sense to the stochastic neural field equation.

Our starting point is the random field solution, given by Theorem \ref{thm: e+u GN}.  Suppose that the conditions of Theorem \ref{thm: e+u GN} are satisfied (i.e. $\varphi\in L^2(\R^N)$, $\sigma:\R\to\R$ Lipschitz, $G:\R\to\R$ Lipschitz and bounded, $w$ satisfies (\textbf{C2'}) and the given assumptions on the initial condition).
%i.e. $\varphi\in L^2(\R^N)$, $\sigma:\R\to\R$ is Lipschitz, $G:\R\to\R$ is Lipschitz and bounded, $w$ satisfies (\textbf{C2'}), the map $x\mapsto Y_0(x)$ is Borel measurable, $Y_0(x)$ is $\cF_0$-measurable for all $x\in\R^N$ and 
%\begin{equation}
%\label{initial condition}
%\sup_{x\in{\R^N}}\E\left[|Y_0(x)|^2\right] < \infty.
%\end{equation}
Then, by that result, there exists a unique random field $(Y(t, x))_{t\geq0, x\in\R^N}$ such that
\begin{align}
\label{Walsh solution}
Y(t, x) & = e^{-t}Y_0(x) + \int_0^te^{-(t-s)}\int_{\R^N} w(x,y)G(Y(s,y))dyds \nonumber\\
&\qquad \qquad + \int_0^t\int_{\R^N}e^{-(t-s)}\sigma(Y(s, x))\varphi(x-y)W(ds dy)
\end{align}
such that
\begin{equation}
\label{uniform bound}
\sup_{t\in[0, T], x\in\R^N}\E\left[|Y(t, x)|^2\right] < \infty,
\end{equation}
for all $T>0$, and we say that $(Y(t, x))_{t\geq0, x\in\R^N}$ is the random field solution to the stochastic neural field equation.  

The relationship between this random field solution to the stochastic neural field equation, and a solution constructed as a Hilbert space valued process according to Section \ref{EE} is given in the following theorem.

\begin{thm}
\label{thm: comparison}
Suppose the conditions of Theorem \ref{thm: e+u GN} and Theorem \ref{thm: Lp boundedness} are satisfied.  Moreover suppose that {\rm (\textbf{C1'})} is satisfied for some $\rho_w\in L^1(\R^N)$, and fix $U = L^2(\R^N)$ and $H= L^2(\R^N, \rho_w)$. 
 Then the random field $(Y(t, x))_{t\geq0}$ satisfying \eqref{Walsh solution} and \eqref{uniform bound} is such that $(Y(t))_{t\geq0} := (Y(t, \cdot))_{t\geq0}$ is the unique mild $H$-valued solution to the stochastic evolution equation
\begin{equation}
\label{DZ interp}
dY(t) = [-Y(t) +  \mathbf{F}(Y(s))]dt + \mathbf{B}(Y(t))dW(t),\qquad t\in[0, T],
\end{equation}
where $(W(t))_{t\geq0}$ is a $U$-valued $Q$-Wiener process, $\mathbf{B}:H \to L_0(U, H)$ is given by
\[
\mathbf{B}(h)(u)(x) := \int_{\R^N}\sigma(h(x))\varphi(x-y)u(y)dy, \qquad h \in H,\ u\in U.
\]
and $\mathbf{F}:H \to H$ is given by
\[
\mathbf{F}(h)(x) = \int_{\R^N}w(x,y)G(h(y))dy, \quad x\in \R^N, \ h\in H.
\]
 \end{thm}

 \begin{rem}
 \label{rem: comparison}
 We note that the evolution equation \eqref{DZ interp} is subtly different to \eqref{NF SEE} considered in Section \ref{Application}, in that the noise term added here is not generally of the form $\tilde{\sigma}(Y(t))\circ BdW(t)$ for some $\tilde{\sigma}: H \to L_0(H)$, with $B:U\to H$ given by
 \[
 B(u)(x) = \int_{\R^N}\tilde{\varphi}(x-y)u(y)dy, \quad x\in\R^N, \ u\in U,
 \]
 for some $\tilde{\varphi} \in L^1(\R^N)$.
 Indeed, even if the $\varphi$ we consider in \eqref{Walsh solution} is also in $L^1(\R^N)$, in order for the noise term $\mathbf{B}(Y(t))dW(t)$ in \eqref{DZ interp} to have this structure, since $\mathbf{B}(h)(u)(x) = \sigma(h(x))B(u)(x)$ for all $h\in H$, $u\in U$ and $x\in\R^N$, we would require $\tilde{\sigma}$ to be defined by
\[
\tilde{\sigma}(h_1)(h_2)(x) = \sigma(h_1(x))h_2(x), \qquad \forall x\in\R^N,\ h_1, h_1\in H.
\]
This is not a well-defined map $H \to L_0(H)$ unless $\sigma$ is bounded.  

However, in the case when $\sigma$ is bounded (for example $\sigma \equiv 1$), under the assumptions that $\varphi\in L^1(\R^N)\cap L^2(\R^N)$ and $w$ satisfies {\rm (\textbf{C1'})}, the above result shows that the random field solution is a special case of the $H$-valued solution to \eqref{NF SEE} constructed in Section \ref{Application}, where the noise term is given by $\tilde{\sigma}(Y(t))\circ BdW(t)$.
 \end{rem}

\begin{proof}[Proof of Theorem \ref{thm: comparison}]

We first check that \eqref{DZ interp} does indeed have a mild solution according to Theorem \ref{D-Z:e+u}.  This does not follow directly from Theorem \ref{NF SEE:e+u} (see Remark \eqref{rem: comparison}).  However, under the current assumptions, it is easy to check that conditions (H1) - (H4) of Theorem \ref{D-Z:e+u} are satisfied.  Indeed, we can follow most of the proof of Theorem  \ref{NF SEE:e+u}, inserting where necessary the facts that
$\mathbf{B}:H \to L_0(U,H)$ is well-defined, since for any $h\in H$, $u\in U$,
\begin{align*}
\|\mathbf{B}(h)(u)\|^2_{H} &=  \int_{\R^N}\sigma^2(h(x))\left(\int_{\R^N}\varphi(x-y)u(y)dy\right)^2\rho_w(x)dx\\
&\leq C_\sigma^2\|\varphi\|_{L^2(\R^N)}^2 \|u\|^2_U(\|\rho_w\|_{L^1(\R^N)} + \|h\|_{H}^2),
\end{align*}
and moreover that (see Example \ref{bounded is HS})
\begin{align*}
\|\mathbf{B}(h_1) - \mathbf{B}(h_2)\|^2_{L_2(Q^\frac{1}{2}(U), H)} &\leq \mathrm{Tr}(Q) \|\mathbf{B}(h_1) - \mathbf{B}(h_2)\|^2_{L_0(U, H)}\\
& \leq  \mathrm{Tr}(Q)C_\sigma^2\|\varphi\|_{L^2(\R^N)}^2\|h_1 - h_2\|_{H}^2,
\end{align*}
for all $h_1, h_2\in H$.  Finally, by the assumptions on the initial condition, and since $\rho_w\in L^1(\R^N)$, we see that $Y_0\in H$ is $\cF_0$-measurable and $\mE{\|Y_0\|_H^p}<\infty$ for all $p\geq2$.  Hence we can apply Theorem \ref{D-Z:e+u}.

The proof of the result involves some technical definition chasing, and in fact is contained in \cite{Dalang}, though rather implicitly.  It is for this reason that we carry out the proof explicitly in our situation, by closely following \cite[Proposition 4.10]{Dalang}. The most important point is to relate the stochastic integrals that appear in the two different formulations of a solution. To this end, define
\[
\cI(t, x) :=  \int_0^t\int_{\R^N}e^{-(t-s)}\sigma(Y(s, x))\varphi(x-y)W(ds dy), \quad x\in\R^N,\ t\geq0,
\]
to be the Walsh integral that appears in the random field solution \eqref{Walsh solution}.  Our aim is to show that
\begin{equation}
\label{to prove}
\cI(t, \cdot) = \int_0^te^{-(t-s)}\mathbf{B}(Y(s))dW(s),
\end{equation}
where the integral on the right-hand side is the $H$-valued stochastic integral which appears in the mild formulation of a solution to \eqref{DZ interp}, defined according to Definition \ref{defin: DZ solution}.  

\vspace{0.5cm}
\noindent\textit{Step 1:}
Adapting Proposition 2.6 of \cite{Dalang} very slightly, we have that the Walsh integral $\cI(t, x)$ can be written as the integral with respect to the \textit{cylindrical} Wiener process $\cW = \{\cW_t(u): t\geq 0, u\in U\}$ with covariance $\mathbf{Id}_U$.
\footnote{This is a family of random variables such that for each $u\in U$, $(\cW_t(u))_{t\geq0}$ is a Brownian motion with variance $t\|u\|^2_U$, and for all $s, t\geq0$, $u_1,u_2\in U$, $\E[\cW_t(u_1)\cW_s(u_2)] = (s\wedge t)\langle u_1, u_2\rangle_U$. See for example \cite{Dalang} Section 2.1}
Precisely, we have
\[
\mathcal{I}(t, x) = \int_0^t g^{t, x}_sd\cW_s,
\]
for all $t\geq 0, x\in\R^N$, where $g^{t,x}_s(y) := e^{-(t-s)} \sigma(Y(s, x))\varphi(x-y)$, $y\in\R^N$, which is in  $L^2(\Omega \times[0,T]; U)$ for any $T>0$ thanks to \eqref{uniform bound}.
By definition, the integral with respect to the cylindrical Wiener process $\cW$ is given by 
\[
\int_0^t g^{t, x}_sd\cW_s = \sum_{k=1}^\infty \int_0^t \langle g^{t, x}_s, e_k\rangle_Ud\beta_k(s),
\]
where $\{e_k\}_{k=1}^\infty$ is a complete orthonormal basis for $U$, and $(\beta_k(t))_{t\geq0} := (\cW_t(e_k))_{t\geq0}$ are independent real-valued Brownian motions.  This series is convergent in $L^2(\Omega)$.

%We note that this sequence is convergent in $L^2(\Omega)$ since
%\[
%\E\left [\sum_{k=1}^\infty \int_0^t \langle g^{t, x}_s, e_k\rangle_Ud\beta_k(s)\right]^2 = \sum_{k=1}^\infty \int_0^t \langle g^{t, x}_s, e_k\rangle_U^2ds = \int_0^t  \|g^{t, x}_s\|^2_U ds.
%\]
%

\vspace{0.5cm}
\noindent\textit{Step 2:}
Fix arbitrary $T>0$. As in Section 3.5 of \cite{Dalang}, we can consider the process $\{W(t), t\in[0, T]\}$ defined by 
\begin{equation}
\label{cylindrical Wiener}
W(t) = \sum_{k=1}^\infty \beta_k(t) J(e_k)
\end{equation}
%\olivier{{\bf Olivier}: I suggest to get rid of the confusing $U_1$}\\
where $J:U\to U$ is a Hilbert-Schmidt operator. % and $U_1$ is another Hilbert space, possibly larger than $U$.  
%We say that $W$ is a cylindrical $\mathbf{Id}_U$-Wiener process on $U$.  
$W(t)$ takes its values in $U$, where it is a $Q (= JJ^*)$-Wiener process with $\mathrm{Tr}(Q)<\infty$ (Proposition 3.6 of \cite{Dalang}). 
We define $J(u) := \sum_{k} \sqrt{\lambda_k}\langle u, e_k\rangle_Ue_k$ for a sequence of positive real numbers $(\lambda_k)_{k\geq1}$ such that $\sum_k \lambda_k <\infty$.

Now define
\begin{align*}
\Phi_s^{t,x}(u) &= \left\langle g^{t, x}_s, u \right\rangle_U,
\end{align*}
which takes values in $\R$.  
Proposition 3.10 of \cite{Dalang} tells us that the process $\{\Phi_s^{t,x}, s\in[0, T]\}$ defines a predictable process with values in $L_2(U, \R)$ and
\begin{equation}
\label{DZ link}
\int_0^t \Phi_s^{t,x} dW(s) = \int_0^tg^{t, x}_sd\cW_s,
\end{equation}
where the integral on the left is defined according to Section \eqref{int Q-Wiener}, with values in $\R$.

\vspace{0.5cm}
\noindent\textit{Step 3:}
We now note that the original Walsh integral $\cI(\cdot, \cdot) \in L^2(\Omega \times [0, T]; H)$.  Indeed, because of \eqref{L2 burkholder},
\begin{align*}
\|\cI \|_{ L^2(\Omega \times [0, T]; H)}^2 &= \E\left[\int_0^T\|\cI(t,\cdot)\|_H^2dt\right] = \int_0^T\int_{\R^N}\mE{|\cI(t,x)|^2}\rho_w(x)dxdt\\
& \leq  \|\varphi\|_{L^2(\R^N)}^2 \int_0^T\int_0^t\int_{\R^N}e^{-2(t-s)}\mE{\sigma^2(Y(s, x))}ds \rho_w(x)dxdt< \infty,
\end{align*}
again thanks to \eqref{uniform bound}.
Hence $\cI(t, \cdot)$ takes values in $H$, and we can therefore write
\begin{align*}
\cI(t, \cdot) &= \sum_{j=1}^\infty\langle\cI(t, \cdot), f_j \rangle_Hf_j= \sum_{j=1}^\infty\left\langle\int_0^t \Phi_s^{t,\cdot} dW(s), f_j \right\rangle_Hf_j,
\end{align*}
by \eqref{DZ link}, where $\{f_j\}_{j=1}^\infty$ is a complete orthonormal basis in $H$.  Moreover, by using \eqref{cylindrical Wiener}
\begin{align}
\label{expansion}
\cI(t, \cdot) &=  \sum_{j=1}^\infty\left(\int_{\R^N}\left(\int_0^t \Phi_s^{t,x} dW(s)\right) f_j(x) \rho_w(x)dx\right)f_j\nonumber\\
&=  \sum_{j=1}^\infty\left(\int_{\R^N}\left(\sum_{k=1}^\infty \int_0^t \Phi_s^{t,x}(\sqrt{\lambda_k}e_k) d\beta_k(s)\right) f_j(x) \rho_w(x)dx\right)f_j.
\end{align}

Finally, consider the $H$-valued stochastic integral 
\[
\int_0^te^{-(t-s)}\mathbf{B}(Y(s))dW(s),
\]
where $\mathbf{B}: H \to L_0(U, H)$ is given above.
Then similarly 
\begin{align*}
&\int_0^te^{-(t-s)}\mathbf{B}(Y(s))dW(s) = \sum_{j=1}^\infty \left\langle\int_0^te^{-(t-s)}\mathbf{B}(Y(s))dW(s), f_j\right\rangle_Hf_j\\
&\qquad= \sum_{j=1}^\infty \left\langle\sum_{k=1}^\infty \int_0^te^{-(t-s)}\sqrt{\lambda_k}\mathbf{B}(Y(s))(e_k)d\beta_k(s), f_j\right\rangle_Hf_j\\
&\qquad = \sum_{j=1}^\infty\left( \int_{\R^N}\left(\sum_{k=1}^\infty \int_0^te^{-(t-s)}\sqrt{\lambda_k}\mathbf{B}(Y(s))(e_k)(x)d\beta_k(s)\right)f_j(x)\rho_w(x)dx\right)f_j.
\end{align*}
Here, by definition, for $x\in\R^N$, $0\leq s \leq t$,
\begin{align*}
e^{-(t-s)}\sqrt{\lambda_k}\mathbf{B}(Y(s))(e_k)(x) &= \int_{\R^N}e^{-(t-s)}\sigma(Y(s, x))\varphi(x-y)\sqrt{\lambda_k}e_k(y)dy \\
&=e^{-(t-s)}\sigma(Y(s, x)) \langle \varphi(x-\cdot),\sqrt{\lambda_k}e_k\rangle_U =\Phi_s^{t,x}(\sqrt{\lambda_k}e_k),
\end{align*}
which proves \eqref{to prove} by comparison with \eqref{expansion}.

\vspace{0.5cm}
\noindent\textit{Step 4:}
To conclude it suffices to note that the pathwise integrals in \eqref{Walsh solution} and the mild $H$-valued solution to \eqref{DZ interp} coincide as elements of $H$.  Indeed, it is clear that, by definition of $\mathbf{F}$,
\[
 \int_0^te^{-(t-s)}\int_{\R^N} w(\cdot,y)G(Y(s,y))dyds = \int_0^t e^{-(t-s)}\mathbf{F}(Y(s))ds,
\]
where the later in an element of $H$.

\end{proof}

\section{Conclusion}
\label{conclusion}
We have here explored two alternative ways to define in a mathematically precise fashion the notion of a stochastic neural field. Both of these approaches have been used previously without theoretical justification by scientists in the field of theoretical neuroscience.  Indeed, the approach of using the theory of Hilbert space valued processes presented by Da Prato and Zabczyk  (analysed in Section \ref{EE}) is adopted in \cite{RiedlerKuehn}, while we argue the random field approach is that which is implicitly used by Bressloff, Ermentrout and their associates in \cite{bressloff-webber-2012,bressloff-wilkerson:12,kilpatrick-ermentrout:13}. 

The difference between the two constructions is completely determined by the type of noise that one wishes to consider in the neural field equation, which may give rise to inherently different solutions.  The advantage of the construction of a solution as a stochastic process taking values in a Hilbert space carried out in Section \ref{EE}, is that it allows one to consider more general diffusion coefficients (see Remark \ref{rem: comparison}).  Moreover, our construction using this approach can also handle a noise term that has no spatial correlation i.e. a pure space-time white noise, by taking the correlation function $\varphi$ to be a Dirac mass (see Section \ref{color}).  A disadvantage is that we have to be careful to impose conditions which control the behavior of the solution in space at infinity and guarantee the integrability of the solution.  In particular we require that the connectivity function $w$ either satisfies the strong conditions (\textbf{C1}) and (\textbf{C2}), or the weaker but harder to check conditions (\textbf{C1'}) and (\textbf{C2'}).

On the other hand, the advantage of the random field approach developed in Section \ref{random fields} is that one no longer needs to control what happens at infinity.  We therefore require fewer conditions on the connectivity function $w$ to ensure the existence of a solution ((\textbf{C2'}) is sufficient -- see Theorem \ref{thm: e+u GN}).  Moreover, with this approach, it is easier to write down conditions that guarantee the existence of a solution that is continuous in both space and time (as opposed to the Hilbert space approach, where spatial regularity is somewhat hidden).  However, in order to avoid non-physical distribution valued solutions, we had to impose \textit{a priori} some extra spatial regularity on the noise (see Section \ref{smoothed noise}).

The relationship between the two approaches is summarized in Section \ref{sec: comparison}, where we showed that if we impose the extra condition (\textbf{C1'}) to ensure integrability, it is possible to reinterpret the random field solution as a Hilbert space valued process that satisfies an infinite dimensional stochastic evolution equation, though with a subtly different noise term to those equations originally considered.  Nonetheless, we are able to see that when $\sigma:\R\to\R$ is bounded, the random field solution is in fact a special case of the Hilbert space valued solution constructed in Section \ref{Application} (under the additional condition that $\varphi\in L^1(\R^N)\cap L^2(\R^N)$ -- see Remark \ref{rem: comparison}).

Our main conclusion here is thus that the approach to take really does depend on the end goal.  If one is interested in very general diffusions, and has some strong decay properties on $w$, the infinite dimensional Hilbert space approach is well-suited.  On the other hand, if one is interested in spatially regular solutions, and does not wish to impose such strong decay properties on $w$, but is content with the addition of less general and more regular noise terms, the random field approach should be taken.

We end with a word about the applicability of our results.  Neural field equations are commonly encountered in neuroscience with regard to modeling brain areas. In practice one is often interested in modeling two- or three-dimensional pieces of cortex whose size is large with respect to that of the support of the connectivity kernel $w$. It is often necessary to extend the physical space where the brain tissues are located with representation spaces to account for the computations performed by the neurons. For example in visual perception features such as disparity (related to depth perception), velocity (related to visual motion perception), or color can be represented by points in $\R^2$ and $\R^3$. In sound perception the local Fourier analysis performed by the cochlea is represented by a spatial distribution of points in $\mathbb{C}$. Other examples can be found in the motor cortex where the neurons preparing for an action of part of the body store representations for driving the effector muscles that are naturally represented by points in some $\R^N$.

 If the size of the connectivity kernel becomes comparable to that of the considered brain area or if the feature space is naturally bounded (e.g. visual orientations, rotation angles for effectors) it becomes more natural to work with a bounded subset of $\R^N$ with periodic or zero boundary conditions.  However, both our approaches still apply in this setup (and are in fact easier to justify).
 
\vspace {0.7cm}
\noindent
{\bf Acknowledgements}\\
The authors are grateful to James Maclaurin for suggesting the use of the Fourier transform in Example 2 on page \pageref{page:example2}.
\bibliographystyle{siam}
\bibliography{NeuralField}

\end{document}